%% file: Jung.tex
\documentclass[final,hidelinks,onefignum,onetabnum]{siamart251216}
\input{shared_info}

\hypersetup{
  colorlinks=false, pdfborder={0 0 1}, linkbordercolor={1 0 0}, citebordercolor={0 1 0} 
}

\begin{document}

\maketitle

\begin{abstract}
    This paper deals with nonsmooth convex optimization problems in Euclidean spaces.
    We identify special elements of the subdifferential of a convex function, called specular gradients.
    Based on this observation, we propose three numerical methods that use specular gradients in subgradient methods.
    We prove the convergence of the proposed methods under suitable step sizes.
    Numerical experiments demonstrate that the proposed methods are capable of minimizing non-differentiable functions that classical methods fail to minimize.
\end{abstract}



\section{Introduction}

Consider the unconstrained optimization problem
\begin{equation} \label{OPT:unconstrained_convex_opt}
    \min_{\mathbf{x} \in \mathbb{R}^n} f(\mathbf{x}),
\end{equation}
where $f: \mathbb{R}^n \to \mathbb{R}$ is convex.
A standard numerical approach for solving this problem is the \emph{subgradient} method, which generates a sequence $\left\{ \mathbf{x}_k \right\}_{k=0}^{\infty}$ according to the formula 
\begin{displaymath}
    \mathbf{x}_{k+1} = \mathbf{x}_k - h_k \, \mathbf{s}_k
\end{displaymath}
for each $k \in \mathbb{N} \cup \left\{ 0 \right\}$, where $h_k > 0$ is a step size, $\mathbf{x}_0 \in \mathbb{R}^n$ is a starting point, and $\mathbf{s}_k \in \partial f(\mathbf{x}_k)$ is a subgradient. 
The gradient descent method chooses $\mathbf{s}_k = Df(\mathbf{x}_k)$ if the objective function $f$ is differentiable. 

This paper proposes three methods for choosing $\mathbf{s}_k = \sg f(\mathbf{x}_k)$, where $\sg f$ denotes the \emph{specular gradient}.
Specular differentiation generalizes classical differentiation via the angular mean of one-sided difference quotients.
\Cref{fig:sg} illustrates that the specular gradient bisects the angle determined by the one-sided directional derivatives and belongs to the subdifferential of a convex function.
For specular differentiation, we refer the reader to \cite{2026a_Jung} in one-dimensional Euclidean spaces and to \cite{2026b_Jung} in normed vector spaces.

\begin{figure}[h]
    \centering
    \includegraphics[width=0.8\textwidth]{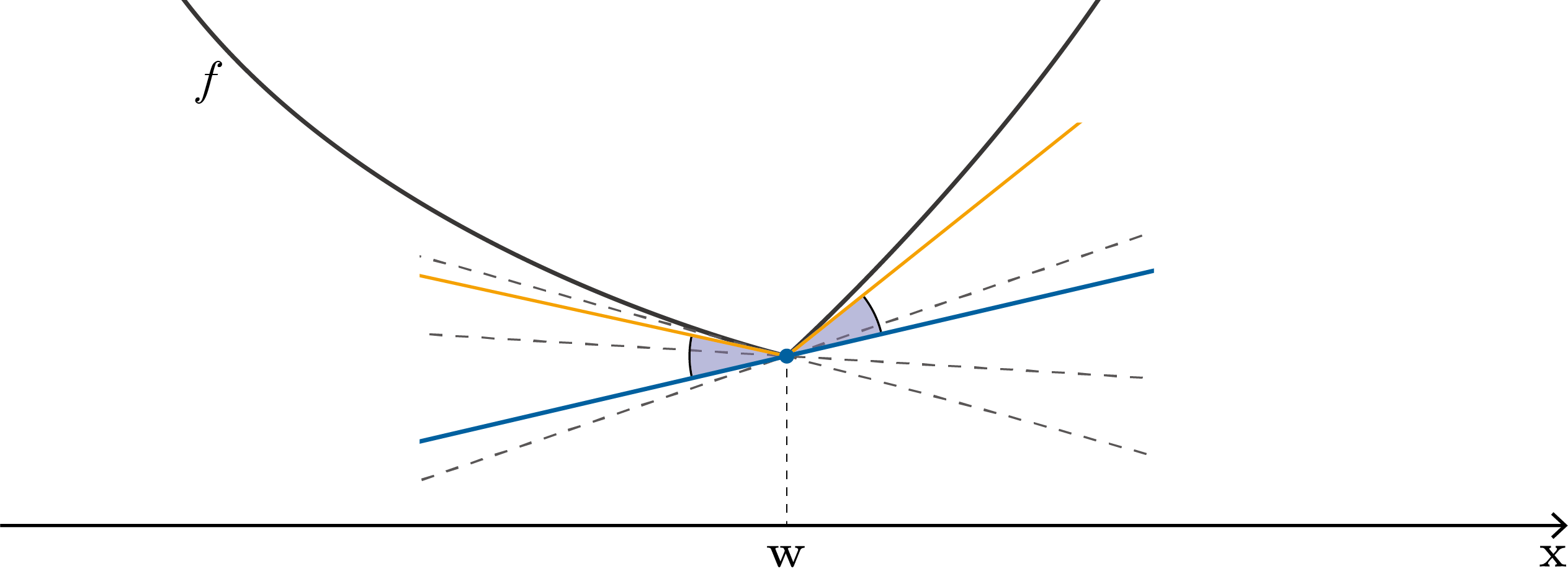}
    \caption{One-sided directional derivatives (yellow lines), the specular gradient (blue line), and subgradients (dotted lines) of a convex function $f$ at $\mathbf{x} = \mathbf{w}$.}
    \label{fig:sg}
\end{figure}

The primary motivation for this choice lies in the computational amenability of specular differentiation.
We show that gradients in the specular sense are subgradients of convex functions in $\mathbb{R}^n$.
This property allows us to bypass the often-expensive computational step of computing an element of the subdifferential.
This computational shortcut is useful in subgradient-based methods for nonsmooth optimization.

A secondary motivation comes from numerical examples in which standard methods, including gradient descent \cite{2004_Boyd_BOOK}, Adam \cite{2017_Kingma}, and BFGS \cite{1970_Broyden,1970_Fletcher,1970_Goldfarb,1970_Shanno}, fail to converge to a minimizer, whereas the proposed methods attain lower objective values.
Thus, one may expect numerical methods based on specular differentiation to exhibit favorable numerical behavior in certain cases.
For example, \cite{2026a_Jung} introduced the \emph{specular ellipse} scheme for ordinary differential equations whose solution trajectories are ellipses, and showed that it has zero local truncation error.

We now describe the three proposed methods in greater detail.
The Python package \texttt{specular-differentiation}, which includes the three proposed methods, is publicly available in \cite{2026s_Jung_SIAM}.

The first method is the \emph{specular gradient} method, which chooses a specific specular gradient as the subgradient direction in general Euclidean spaces; see \cref{alg:SPEG}.
We will show the convergence of the specular gradient method under a specific choice of step size; see \cref{thm: convergence_of_SPEG}.
In $\mathbb{R}^n$, the convergence rate of the specular gradient method is generally sublinear, similar to the classical subgradient method.
This paper focuses on providing motivation for investigating the specular gradient method and proving its convergence, rather than on improving the convergence rate itself.

To improve computational efficiency, the subgradient $\mathbf{s}_k$ can be chosen via a stochastic approach.
Let $\mathbb{E}$ and $\mathbb{P}$ denote the mathematical expectation and the probability, respectively.
Also, let $\mathbb{E}[\, \cdot \,| \, \cdot \,]$ be the conditional expectation.
The stochastic subgradient method chooses $\mathbf{s}_k$ such that $\mathbb{E}[\mathbf{s}_k |\mathbf{x}_k] \in \partial f(\mathbf{x}_k)$.
For convergence of the stochastic subgradient method with square summable but not summable step sizes, see \cite[sect. 2.4]{1998_Shor_BOOK} and \cite{2018_Boyd}.
For diverse objective functions, we refer to \cite[Chap. 7]{2023_Ryu_BOOK}.
For differentiable and $\lambda$-strongly convex objective functions, see \cite[sect. 19.3]{2022_Foucart_BOOK}.

By choosing the specular gradient with a stochastic approach as before, we devise the \emph{stochastic specular gradient} method; see \cref{alg:S-SPEG}.
We prove the convergence of the stochastic specular gradient method with a specific choice of step size; see \cref{thm: convergence_of_S_SPEG}.
To improve the computational efficiency, we suggest the \emph{hybrid specular gradient} method, which is a combination of the specular gradient and stochastic specular gradient methods.

For the step sizes, we follow standard practice for subgradient methods.
Various step-size strategies that ensure convergence of subgradient methods can be adopted: constant step sizes (\cite[Thm. 2.1]{1985_Shor_BOOK}), non-summable diminishing step sizes (\cite[Thm. 2.2]{1985_Shor_BOOK}, \cite[sect. 3.2.3]{2018_Nesterov_BOOK}), square-summable but not summable step sizes (\cite[Thm. 7.4]{2006_Ruszczynski_BOOK}), and Polyak's step size (\cite[sect. 5.3.2]{1987_Polyak_BOOK}, \cite[Chap. 4]{2014_Boyd}). 
It is well known that the subgradient method typically has a sublinear convergence rate. 
If additional assumptions are imposed, the subgradient method can achieve a linear convergence rate; see \cite[sect. 2.3]{1985_Shor_BOOK}.

The main results of this paper are summarized as follows.
We show that the specular gradient belongs to the subdifferential of a proper convex function; see \cref{thm:sg_is_subd}.
We prove the convergence of the specular gradient method (\cref{thm: convergence_of_SPEG}) and the stochastic specular gradient method (\cref{thm: convergence_of_S_SPEG}) with square-summable but not summable step sizes.

\subsection{Notations and Preliminaries}

The \emph{Euclidean norm} and the \emph{$\ell^1$-norm} are denoted by 
\begin{displaymath}
    \left\| \mathbf{x} \right\|_{\mathbb{R}^n} := \left( \sum_{i=1}^n x_i^2 \right)^{\frac{1}{2}} 
    \qquad\text{and}\qquad
    \left\| \mathbf{x} \right\|_{\ell^1} := \sum_{i=1}^n |x_i|
\end{displaymath}
for $\mathbf{x} = (x_1, x_2, \ldots, x_n) \in \mathbb{R}^n$, respectively.
Also, $\innerprd$ denotes the standard \emph{dot product}, $\mathbf{x} \innerprd \mathbf{w} = \sum_{i=1}^n x_i w_i$ for $\mathbf{x} = (x_1, x_2, \ldots, x_n)$, $\mathbf{w} = (w_1, w_2, \ldots, w_n) \in \mathbb{R}^n$.
Let $\left\{ \mathbf{e}_i \right\}_{i=1}^n$ denote the standard basis of $\mathbb{R}^n$.
Denote by $\mathbb{M}^{m \times n}$ the set of all $m \times n$ real matrices.
We write $\mathbb{R}^n = \mathbb{M}^{n \times 1}$.

Suppose that $f : \Omega \to \mathbb{R}$ is a function, where $\Omega$ is an open subset of $\mathbb{R}^n$.
Let $\mathbf{x} \in \Omega$ be a point.
We write the \emph{one-sided directional derivatives} of $f$ at $\mathbf{x}$ in the direction of $\mathbf{v} \in \mathbb{R}^n$ as follows:
\begin{displaymath}
    \partial^+_{\mathbf{v}} f(\mathbf{x}) :=  \lim_{h \searrow 0} \dfrac{f(\mathbf{x} + h\mathbf{v}) - f(\mathbf{x})}{h}
    \qquad\text{and}\qquad
    \partial^-_{\mathbf{v}} f(\mathbf{x}) :=  \lim_{h \searrow 0} \dfrac{f(\mathbf{x}) - f(\mathbf{x} - h\mathbf{v})}{h}.
\end{displaymath}

The \emph{specular directional derivative} of $f$ at $\mathbf{x}$ in the direction of $\mathbf{v} \in \mathbb{R}^n \setminus \left\{ \mathbf{0} \right\}$ is defined by
\begin{multline*} \label{eq:spG}
  \partial_{\mathbf{v}}^{\sd}f(\mathbf{x}) := \lim_{h \searrow 0} \left[ \left( \dfrac{f(\mathbf{x} + h\mathbf{v}) - f(\mathbf{x})}{h} \right) \dfrac{\left\| U \right\|_{\mathbb{R}^{n+1}}}{\left\| U \right\|_{\mathbb{R}^{n+1}} + \left\| V \right\|_{\mathbb{R}^{n+1}}} \right. \\ 
  + \left. \left( \dfrac{f(\mathbf{x}) - f(\mathbf{x} - h\mathbf{v})}{h} \right) \dfrac{\left\| V \right\|_{\mathbb{R}^{n+1}}}{\left\| U \right\|_{\mathbb{R}^{n+1}} + \left\| V \right\|_{\mathbb{R}^{n+1}}} \right],
\end{multline*}
where $U := (h\mathbf{v}, f(\mathbf{x}) - f(\mathbf{x} - h\mathbf{v}))$ and $V := (h\mathbf{v}, f(\mathbf{x} + h\mathbf{v}) - f(\mathbf{x}))$ for sufficiently small $h > 0$.    
If $\mathbf{v} = \mathbf{0}$, we define  $\partial_{\mathbf{v}}^{\sd}f(\mathbf{x}) = 0$.
The LaTeX macro for the symbol $\sd$ is available in \cite{2026s_Jung_SIAM}.
For the derivation of the formula in the definition of $\partial^{\sd}_{\mathbf{v}} f(\mathbf{x})$, we refer the reader to \cite[sect. 2]{2026b_Jung}.

If $\mathbf{v} = \mathbf{e}_i$ for $1 \leq i \leq n$, then we write the \emph{specular partial derivative} of $f$ at $\mathbf{x} = (x_1, x_2, \ldots, x_n)$ with respect to $x_i$ as 
\begin{displaymath}
    \dfrac{\partial^{\sd}}{\partial x_i} f(\mathbf{x}) := \partial_{x_i}^{\sd}f(\mathbf{x}) := \partial_{\mathbf{e}_i}^{\sd}f(\mathbf{x}).
\end{displaymath}
If $n = 1$, the \emph{specular derivative} of $f$ at $x \in \Omega$ is denoted by
\begin{displaymath}    
    \dfrac{d^{\sd}}{dx} f (x)  := f^{\sd}(x) := \partial_1^{\sd} f(x).
\end{displaymath}

We say $f$ is \emph{specularly G\^ateaux differentiable} at $\mathbf{x}$ if $\partial_{\mathbf{v}}^{\sd}f(\mathbf{x})$ exists for all $\mathbf{v} \in \mathbb{R}^n$ and there exists $\ell : \mathbb{R}^n \to \mathbb{R}$ such that $\ell(\mathbf{v}) = \partial_{\mathbf{v}}^{\sd}f(\mathbf{x})$ for all $\mathbf{v} \in \mathbb{R}^n$.
If such an operator $\ell$ exists, we call it the \emph{specular G\^ateaux derivative} of $f$ at $\mathbf{x}$ and write $\sGd f(\mathbf{x}) := \ell$.
A specular G\^ateaux derivative is unique if it exists; see \cite[Prop. 4.1]{2026b_Jung}.

We say $f$ is \emph{specularly Fr\'echet differentiable} at $\mathbf{x}$ if there exists a continuous linear operator $\ell :\mathbb{R}^n \to \mathbb{R}$ such that 
\begin{multline*}  \label{eq: specularly Frechet differentiability}
  \lim_{\left\| \mathbf{w} \right\|_{\mathbb{R}^n} \to 0} \left\vert \left( \dfrac{f(\mathbf{x} + \mathbf{w}) - f(\mathbf{x}) - \ell(\mathbf{w}) }{\left\| \mathbf{w} \right\|_{\mathbb{R}^n}} \right) \dfrac{\left\| J \right\|_{\mathbb{R}^{n+1}}}{\left\| J \right\|_{\mathbb{R}^{n+1}} + \left\| K \right\|_{\mathbb{R}^{n+1}}}\right. \\
  \left. +  \left( \dfrac{f(\mathbf{x}) - f(\mathbf{x} - \mathbf{w}) - \ell(\mathbf{w})}{\left\| \mathbf{w} \right\|_{\mathbb{R}^n}} \right) \dfrac{\left\| K \right\|_{\mathbb{R}^{n+1}}}{\left\| J \right\|_{\mathbb{R}^{n+1}} + \left\| K \right\|_{\mathbb{R}^{n+1}}}  \right\vert  = 0,
\end{multline*}
where $J := (\mathbf{w}, f(\mathbf{x}) - f(\mathbf{x} - \mathbf{w}))$ and $K := (\mathbf{w}, f(\mathbf{x} + \mathbf{w}) - f(\mathbf{x}))$ for $\mathbf{w} \in \mathbb{R}^n$.
If such an operator $\ell$ exists, we call it the \emph{specular Fr\'echet differential} of $f$ at $\mathbf{x}$ and write $\widehat{D} f(\mathbf{x}) := \ell$.
A specular Fr\'echet differential is unique if it exists; see \cite[Prop. 5.1]{2026b_Jung}.

If $f$ is specularly Fr\'echet differentiable at $\mathbf{x}$, the \emph{specular gradient} of $f$ at $\mathbf{x}$ is the unique vector $\sg f(\mathbf{x}) \in \mathbb{R}^n$ such that 
\begin{equation}    \label{def:sg}
    \left\langle \widehat{D} f(\mathbf{x}), \mathbf{w} \right\rangle = \left\langle \sg f(\mathbf{x}), \mathbf{w} \right\rangle_{\mathbb{R}^n}
\end{equation}
for all $\mathbf{w} \in \mathbb{R}^n$.
The LaTeX macro for the symbol $\sg$ is available in \cite{2026s_Jung_SIAM}.

The specular gradient can be written in terms of specular partial derivatives, as is the case with the classical gradient.
Moreover, a weaker version of the necessary condition for optimality can be obtained from the Quasi-Fermat Theorem in the specular sense \cite[Thm. 4.6, Thm. 5.6]{2026b_Jung}.

\begin{theorem} \label{thm:sp_R^n}
    Let $f : \Omega \to \mathbb{R}$ be a specularly Fr\'echet differentiable function, where $\Omega$ is an open subset of $\mathbb{R}^n$.
    \begin{enumerate}[label=\upshape(\alph*)]
        \item \label{thm:sp_R^n-1} For each $\mathbf{x} \in \Omega$, $f$ is specularly G\^ateaux differentiable at $\mathbf{x}$ in any direction $\mathbf{v} \in \mathbb{R}^n \setminus \left\{ \mathbf{0} \right\}$ with 
        \begin{displaymath}
            \partial_{\mathbf{v}}^{\sd} f(\mathbf{x}) = \left\langle \widehat{D} f(\mathbf{x}), \mathbf{v} \right\rangle.
        \end{displaymath}
        \item \label{thm:sp_R^n-2} For each $\mathbf{x} \in \Omega$, it holds that 
        \begin{displaymath}
            \sg f(\mathbf{x}) = \left( \partial_{x_1}^{\sd}f(\mathbf{x}), \partial_{x_2}^{\sd}f(\mathbf{x}), \ldots, \partial_{x_n}^{\sd}f(\mathbf{x}) \right).
        \end{displaymath}
        \item \label{thm:sp_R^n-3} If $\mathbf{x}^{\ast} \in \Omega$ is a local minimizer or maximizer of $f$, it holds that
    \begin{equation}    \label{ineq:nec_cond_opt_-1}
        \left\vert \sg f\left(\mathbf{x}^{\ast}\right) \innerprd \mathbf{v} \right\vert \leq \left\| \mathbf{v} \right\|_{\mathbb{R}^n}
    \end{equation}
    for each $\mathbf{v} \in \mathbb{R}^n$.
    In particular, it holds that 
    \begin{equation}   \label{ineq:nec_cond_opt-Rn}
        \left\vert \sum_{i=1}^n \partial^{\sd}_{x_i} f\left(\mathbf{x}^{\ast}\right) \right\vert \leq \sqrt{n}.
    \end{equation}
    \end{enumerate}
\end{theorem}

\begin{proof}
    \Cref{thm:sp_R^n-1} follows from the fact that specular Fr\'echet differentiability implies specular G\^ateaux differentiability; see \cite[Prop. 5.3]{2026b_Jung}.

    By \cref{thm:sp_R^n-1} and the equality \eqref{def:sg}, we have 
    \begin{displaymath}
        \partial^{\sd}_{x_i} f(\mathbf{x}) 
        = \left\langle \widehat{D} f(\mathbf{x}), \mathbf{e}_i \right\rangle 
        =  \sg f(\mathbf{x}) \innerprd \mathbf{e}_i
    \end{displaymath}
    for each $1 \leq i \leq n$. 
    Hence, the $i$-th component of $\sg f(\mathbf{x})$ is the specular partial derivative $\partial_{x_i}^{\sd}f(\mathbf{x})$, proving \cref{thm:sp_R^n-2}.

    The inequality \eqref{ineq:nec_cond_opt_-1} directly follows from \cite[Thm. 5.6]{2026b_Jung} and \cref{thm:sp_R^n-1}.
    To show the inequality \eqref{ineq:nec_cond_opt-Rn}, choose $\mathbf{v} = (1, 1, \ldots, 1) \in \mathbb{R}^n$ in the inequality \eqref{ineq:nec_cond_opt_-1}.
\end{proof}

Let $f : \Omega \to [-\infty, \infty]$ be a function, where $\Omega$ is an open convex subset of $\mathbb{R}^n$.
We say $f$ is \emph{proper} if $\left\{ \mathbf{x} \in \Omega : f(\mathbf{x}) < \infty \right\}$ is nonempty and $f(\mathbf{x}) > -\infty$ for all $\mathbf{x} \in \Omega$.
Let $\mathbf{x} \in \Omega$ be such that $f(\mathbf{x}) < \infty$.
If $f$ is proper and convex, the \emph{subdifferential} of $f$ at $\mathbf{x}$ is the set 
\begin{displaymath}
    \partial f(\mathbf{x}) := \left\{ \mathbf{s} \in \mathbb{R}^n : f(\mathbf{w}) \geq f(\mathbf{x}) + \mathbf{s} \innerprd (\mathbf{w} - \mathbf{x}) \text{ for all } \mathbf{w} \in \Omega \right\}.
\end{displaymath}
The elements of $\partial f(\mathbf{x})$ are called \emph{subgradients} of $f$ at $\mathbf{x}$.
For subdifferentials, we refer the reader to \cite{1990_Clarke_BOOK,2006_Mordukhovich_BOOK,2018_Nesterov_BOOK,1970_Rockafellar_BOOK,1998_Rockafellar}.
We will prove that $\sg f(\mathbf{x}) \in \partial f(\mathbf{x})$; see \cref{thm:sg_is_subd}.

\subsection{Organization}

The remainder of this paper is organized as follows.
\Cref{sec:nonsmooth_convex_analysis} establishes general nonsmooth convex analysis in the specular sense.
We justify the specular gradient methods by proving that the specular gradient of a convex function belongs to its subdifferential.
In \cref{sec:speg}, we define the specular gradient, stochastic specular gradient, and hybrid specular gradient methods for the nonsmooth convex optimization problem \eqref{OPT:unconstrained_convex_opt}. 
\Cref{sec:convergence} proves the convergence of the specular gradient method and the stochastic specular gradient method with the square-summable but not summable step sizes. 
\Cref{sec:numerical_results} provides numerical results for the proposed methods.

\section{Nonsmooth convex analysis} \label{sec:nonsmooth_convex_analysis}

In this section, we discuss specular differentiation for convex and specularly Fr\'echet differentiable functions.

\begin{lemma} \label{lem: convex function specular derivative and one sided derivatives}
    Let $\Omega$ be an open convex set in $\mathbb{R}^n$.
    If a functional $f : \Omega \to \mathbb{R}$ is convex and specularly Fr\'echet differentiable in $\Omega$, then for each $\mathbf{x} \in \Omega$, it holds that 
    \begin{equation} \label{ineq:conv_sd_odd}
        f(\mathbf{x}) - f(\mathbf{x} - \mathbf{v}) \leq \partial^-_\mathbf{v} f(\mathbf{x}) \leq \partial^{\sd}_\mathbf{v} f(\mathbf{x}) \leq \partial^+_\mathbf{v} f(\mathbf{x}) \leq f(\mathbf{x} + \mathbf{v}) - f(\mathbf{x})
    \end{equation}
    for each $\mathbf{v} \in \mathbb{R}^n$ such that $\mathbf{x} \pm \mathbf{v} \in \Omega$.
\end{lemma}

\begin{proof}
    Fix $\mathbf{x} \in \Omega$.
    The case $\mathbf{v} = \mathbf{0}$ is trivial.
    Thus, let $\mathbf{v} \in \mathbb{R}^n$ be such that $\mathbf{v} \neq \mathbf{0}$.
    By \cref{thm:sp_R^n-1} of \cref{thm:sp_R^n}, $f$ is specularly G\^ateaux differentiable at $\mathbf{x}$, and hence $\partial_\mathbf{v}^{\sd}f(\mathbf{x})$ exists. 
    By \cite[Thm. 23.1]{1970_Rockafellar_BOOK}, $\partial^+_\mathbf{v} f(\mathbf{x})$ and $\partial^-_\mathbf{v} f(\mathbf{x})$ exist with 
    \begin{displaymath}
        -\infty < \partial^-_\mathbf{v} f(\mathbf{x}) = -\partial^+_{-\mathbf{v}} f(\mathbf{x}) \leq \partial^+_\mathbf{v} f(\mathbf{x}) < \infty,
    \end{displaymath}
    where the finiteness follows from the locally Lipschitz continuity of $f$.
    Recall that 
    \begin{displaymath}
        \partial^+_\mathbf{v} f(\mathbf{x}) = \inf_{h > 0} \dfrac{f(\mathbf{x} + h\mathbf{v}) - f(\mathbf{x})}{h} \leq f(\mathbf{x} + \mathbf{v}) - f(\mathbf{x});
    \end{displaymath}
    see, for instance, \cite[Thm. 4.1.3]{2007_Schirotzek_BOOK}.
    Combining these inequalities yields
    \begin{displaymath}
        -\partial^-_\mathbf{v} f(\mathbf{x}) = \partial^+_{-\mathbf{v}} f(\mathbf{x}) \leq f(\mathbf{x} - \mathbf{v}) - f(\mathbf{x}),
    \end{displaymath}
    which implies that 
    \begin{displaymath}
        f(\mathbf{x}) - f(\mathbf{x} - \mathbf{v}) \leq \partial^-_\mathbf{v} f(\mathbf{x}).
    \end{displaymath}
    Similarly, $\partial^-_{\mathbf{v}} f(\mathbf{x}) \leq f(\mathbf{x} + \mathbf{v}) - f(\mathbf{x})$.
    Thus, it is enough to prove the middle two inequalities in \eqref{ineq:conv_sd_odd}.

    We now show the inequalities \eqref{ineq:conv_sd_odd}.
    By \cite[Cor. 3.11]{2026b_Jung}, we have 
    \begin{displaymath}
        \partial^{\sd}_{\mathbf{v}} f(\mathbf{x}) = \left\| \mathbf{v} \right\|_{\mathbb{R}^n} \mathcal{A} \left( \dfrac{\partial^+_\mathbf{v} f(\mathbf{x})}{\left\| \mathbf{v} \right\|_{\mathbb{R}^n}}, \dfrac{\partial^-_\mathbf{v} f(\mathbf{x})}{\left\| \mathbf{v} \right\|_{\mathbb{R}^n}} \right),
    \end{displaymath}
    where the function $\mathcal{A} : \mathbb{R}^2 \to \mathbb{R}$ is defined by
    \begin{displaymath}
        \mathcal{A} (\alpha, \beta) 
        := \left( \frac{\alpha}{\sqrt{1 + \alpha^2}} + \frac{\beta}{\sqrt{1 + \beta^2}} \right) \left( \frac{1}{\sqrt{1 + \alpha^2}} + \frac{1}{\sqrt{1 + \beta^2}} \right)^{-1}.
    \end{displaymath}
    Note that $\mathcal{A}(\alpha, \beta) = 0$ when $\alpha + \beta = 0$; see \cite[App. A]{2026a_Jung} for further analysis of the function $\mathcal{A}$.
    Therefore, it is enough to consider the following two cases.
    On the one hand, if $\partial^+_\mathbf{v} f(\mathbf{x}) + \partial^-_\mathbf{v} f(\mathbf{x}) = 0$, then $\partial^{\sd}_\mathbf{v} f(\mathbf{x}) = 0$, and hence 
    \begin{displaymath}
        \partial^-_\mathbf{v} f(\mathbf{x}) \leq 0 = \partial^{\sd}_\mathbf{v} f(\mathbf{x}) = 0 \leq \partial^+_\mathbf{v} f(\mathbf{x}).
    \end{displaymath}
    On the other hand, if $\partial^+_\mathbf{v} f(\mathbf{x}) + \partial^-_\mathbf{v} f(\mathbf{x}) \neq 0$, then 
    \begin{displaymath}
        \dfrac{\partial^{\sd}_\mathbf{v} f(\mathbf{x})}{\left\| \mathbf{v} \right\|_{\mathbb{R}^n}}=  \mathcal{A} \left( \dfrac{\partial^+_\mathbf{v} f(\mathbf{x})}{\left\| \mathbf{v} \right\|_{\mathbb{R}^n}}, \dfrac{\partial^-_\mathbf{v} f(\mathbf{x})}{\left\| \mathbf{v} \right\|_{\mathbb{R}^n}} \right).
    \end{displaymath}
    Combining this equality with the estimates in \cite[Lem. A.3 (d)]{2026a_Jung} implies the desired inequalities in \cref{ineq:conv_sd_odd}.
\end{proof} 

In \cref{lem: convex function specular derivative and one sided derivatives}, the question of whether the assumption of specular Fr\'echet differentiability can be relaxed is beyond the scope of this paper.
A stronger version of Rademacher's theorem (see, for example, \cite[Thm. 5.2.4]{2025_Niculescu}) would be required.
More precisely, one would need to show that a convex function $f : \mathbb{R}^n \to \mathbb{R}$ is specularly Fr\'echet differentiable at $\mathbf{x} \in E_\mathbf{v}$, where the Lebesgue measurable set $E_\mathbf{v}$ is defined by
\begin{displaymath}
    E_\mathbf{v} := \left\{ \mathbf{x} \in \mathbb{R}^n : \liminf_{h \to 0} \dfrac{f(\mathbf{x} + h\mathbf{v}) - f(\mathbf{x})}{h} < \limsup_{h \to 0} \dfrac{f(\mathbf{x} + h\mathbf{v}) - f(\mathbf{x})}{h}\right\}
\end{displaymath}
for each $\mathbf{v} \in \mathbb{R}^n$.

Therefore, we retain the assumption of specular Fr\'echet differentiability in the following statements. 
The specular gradient of a convex function is a subgradient of the function.

\begin{theorem}  \label{thm:sg_is_subd}
    Let $\Omega$ be an open convex set in $\mathbb{R}^n$, and let $f : \Omega \to \mathbb{R}$ be a function which is convex and specularly Fr\'echet differentiable in $\Omega$.
    Then, for each $\mathbf{x} \in \Omega$, it holds that $\sg f(\mathbf{x}) \in \partial f(\mathbf{x})$, that is, 
    \begin{equation} \label{ineq: specular gradient of convex function is a subgradient}
        f(\mathbf{w}) \geq \sg f(\mathbf{x}) \innerprd (\mathbf{w} - \mathbf{x}) + f(\mathbf{x})  
    \end{equation}
    for all $\mathbf{w} \in \Omega$.
\end{theorem}

\begin{proof}
    Let $\mathbf{v} \in \mathbb{R}^n$ be such that $\mathbf{v} \neq \mathbf{0}$.
    By \cref{lem: convex function specular derivative and one sided derivatives} and \cref{thm:sp_R^n} \ref{thm:sp_R^n-1}, it holds that 
    \begin{equation}    \label{ineq:speg_is_subg}
        \partial_{\mathbf{v}}^+ f(\mathbf{x}) 
        \geq \partial_{\mathbf{v}}^{\sd} f(\mathbf{x}) 
        = \left\langle \widehat{D} f(\mathbf{x}), \mathbf{v} \right\rangle = \sg f(\mathbf{x}) \innerprd \mathbf{v} .
    \end{equation}
    This inequality also holds for $\mathbf{v} = \mathbf{0}$.
    Since the inequality \eqref{ineq:speg_is_subg} holds for any $\mathbf{v} \in \mathbb{R}^n$, the specular gradient $\sg f(\mathbf{x})$ is a subgradient of $f$ at $\mathbf{x}$ by \cite[Thm. 23.2]{1970_Rockafellar_BOOK}.
\end{proof}

The following is a collection of results obtained by a straightforward application of established theorems on subgradients to the specular gradient.

\begin{corollary}  \label{cor:sg_is_subd}
    Let $\Omega$ be an open convex set in $\mathbb{R}^n$, and let $f : \Omega \to \mathbb{R}$ be a functional which is convex and specularly Fr\'echet differentiable in $\Omega$.
    \begin{enumerate}[label=\upshape(\alph*), itemsep=1ex]
        \item \label{cor:sg_is_subd-1} For each $\mathbf{x} \in \Omega$, the function 
        \begin{displaymath}
            \mathbf{w} \mapsto \sg f(\mathbf{x}) \innerprd \mathbf{w} - f(\mathbf{w})
        \end{displaymath}
        on $\Omega$ attains its supremum at $\mathbf{w} = \mathbf{x}$.  
        \item \label{cor:sg_is_subd-2} For each $\mathbf{x} \in \Omega$, it holds that
        \begin{displaymath}
            f(\mathbf{x}) + f^{\ast}\left( \sg f(\mathbf{x}) \right)= \sg f(\mathbf{x}) \innerprd \mathbf{x},
        \end{displaymath}
        where $f^{\ast}$ is the Legendre transform of $f$.
        \item \label{cor:sg_is_subd-3} If $\mathbf{x}^{\ast}$ is a local minimizer of $f$, then 
            \begin{displaymath}
                \sg f(\mathbf{x}) \innerprd (\mathbf{x} - \mathbf{x}^{\ast}) \geq 0,
            \end{displaymath}
            for all $\mathbf{x} \in \Omega$.
        \item \label{cor:sg_is_subd-4} If $\sg f(\mathbf{x}) = \mathbf{0}$, then $\mathbf{x}$ is a global minimizer of $f$.
    \end{enumerate}
\end{corollary}

\begin{proof}
    By applying \cref{thm:sg_is_subd} and \cite[Thm. 23.5]{1970_Rockafellar_BOOK}, \cref{cor:sg_is_subd-1,cor:sg_is_subd-2} follow directly.

    For \cref{cor:sg_is_subd-3}, apply the inequality \eqref{ineq: specular gradient of convex function is a subgradient} with $\mathbf{w} = \mathbf{x}^{\ast}$ to obtain 
    \begin{displaymath}
        0 \geq f(\mathbf{x}^{\ast}) - f(\mathbf{x})\geq \sg f(\mathbf{x}) \innerprd (\mathbf{x}^{\ast} - \mathbf{x}) = - \sg f(\mathbf{x}) \innerprd (\mathbf{x} - \mathbf{x}^{\ast}).
    \end{displaymath}
    Multiplying this inequality by $-1$ concludes the proof of \cref{cor:sg_is_subd-3}.

    \Cref{cor:sg_is_subd-4} directly follows from \cref{thm:sg_is_subd}.
\end{proof}

The converse of \cref{thm:sg_is_subd} may not hold, but the subgradient inequality implies the convexity of the function on a convex open set. 

\begin{proposition} \label{thm: subgradient inequality implies convexity}
    Let $\Omega$ be an open convex set in $\mathbb{R}^n$.
    If a function $f : \Omega \to \mathbb{R}$ is specularly Fr\'echet differentiable in $\Omega$ and satisfies
    \begin{equation} \label{ineq: subgradient inequality}
        f(\mathbf{w}) \geq \sg f(\mathbf{x}) \innerprd (\mathbf{w} - \mathbf{x}) + f(\mathbf{x})
    \end{equation}
    for all $\mathbf{x}, \mathbf{w} \in \Omega$, then $f$ is convex in $\Omega$.
\end{proposition}

\begin{proof}
    Fix $\mathbf{x}, \mathbf{w} \in \Omega$, and let $t \in [0, 1]$.
    Since $\Omega$ is convex, $t \mathbf{x} + (1-t)\mathbf{w} =: \mathbf{z} \in \Omega$.
    Apply the inequality \eqref{ineq: subgradient inequality} with $\mathbf{x}= \mathbf{z}$ to get 
    \begin{displaymath}
        f(\mathbf{z}) 
        \leq f(\mathbf{w}) - \sg f(\mathbf{z}) \innerprd (\mathbf{w} - \mathbf{z})
        = f(\mathbf{w}) + t \sg f(\mathbf{z}) \innerprd (\mathbf{x}-\mathbf{w}).
    \end{displaymath}
    Multiplying this inequality by $1-t$ yields that 
    \begin{equation}    \label{ineq: subgradient inequality implies convexity - 1}
        (1-t)f(\mathbf{z}) \leq (1-t) f(\mathbf{w}) + t(1-t) \sg f(\mathbf{z}) \innerprd (\mathbf{x}-\mathbf{w}).
    \end{equation}
    Similarly, from the inequality \eqref{ineq: subgradient inequality}, one has 
    \begin{displaymath}
        f(\mathbf{z}) 
        \leq f(\mathbf{x}) - \sg f(\mathbf{z}) \innerprd (\mathbf{x} - \mathbf{z})
        = f(\mathbf{x}) - (1-t) \sg f(\mathbf{z}) \innerprd (\mathbf{x} - \mathbf{w}).
    \end{displaymath}
    Multiplying this inequality by $t$ implies 
    \begin{equation} \label{ineq: subgradient inequality implies convexity - 2}
        t f(\mathbf{z}) \leq t f(\mathbf{x}) - t(1-t) \sg f(\mathbf{z}) \innerprd (\mathbf{x} - \mathbf{w}).
    \end{equation}
    Adding the two inequalities \eqref{ineq: subgradient inequality implies convexity - 1} and \eqref{ineq: subgradient inequality implies convexity - 2}, one can obtain that 
    \begin{displaymath}
        f(t \mathbf{x} + (1-t)\mathbf{w}) 
        = f(\mathbf{z})
        = t f(\mathbf{z}) + (1 - t) f(\mathbf{z})
        \leq t f(\mathbf{x}) + (1-t) f(\mathbf{w}).
    \end{displaymath}
    Therefore, $f$ is convex in $\Omega$.
\end{proof}

\section{Specular gradient methods} \label{sec:speg}

We now turn to numerical methods for minimizing convex functions.
The main idea is to select the specular gradient as the subgradient in the subgradient (SG) method, justified by \cref{thm:sg_is_subd}.

\begin{definition}
    Consider the unconstrained optimization problem \eqref{OPT:unconstrained_convex_opt}.
    The \emph{specular gradient} (SPEG) method is an algorithm that generates a sequence $\left\{ \mathbf{x}_k \right\}_{k=0}^{\infty} \subset \mathbb{R}^n$ by the formula
    \begin{equation} \label{mtd: specular gradient method}
        \mathbf{x}_{k+1} = \mathbf{x}_k - h_k \sg f(\mathbf{x}_k)
    \end{equation}
    for each $k \in \mathbb{N} \cup \left\{ 0 \right\}$, where $\left\{ h_k \right\}_{k=0}^{\infty} \subset (0, \infty)$ is a sequence of step sizes and $\mathbf{x}_0 \in \mathbb{R}^n$ is a starting point. 
\end{definition}

Since SG is not a descent method (see \cite[Ex. 7.1]{2006_Ruszczynski_BOOK}), SPEG need not be a descent method.
Thus, it is useful to introduce the following notation.
For each $k \in \mathbb{N} \cup \left\{ 0 \right\}$, choose $\mathbf{x}^{\diamond}_k \in \{\mathbf{x}_0,\ldots,\mathbf{x}_k\}$ such that
\begin{equation}    \label{def: best iteration}
    f(\mathbf{x}^{\diamond}_k) = \min_{0 \leq \ell \leq k} f\left( \mathbf{x}_{\ell} \right).
\end{equation}
Since $f(\mathbf{x}^{\diamond}_k)$ is nonincreasing in $k$, the sequence $\left\{ f(\mathbf{x}^{\diamond}_k) \right\}_{k=0}^{\infty}$ has a limit, possibly equal to $-\infty$, as $k \to \infty$.

Here, we provide the pseudocode of SPEG as in \cref{alg:SPEG}.
Since \cref{thm:sp_R^n-3} of \cref{thm:sp_R^n} is not useful in practice, it is not included in the algorithm.
We will prove that SPEG converges to a minimizer with a proper choice of step sizes: \cref{thm: convergence_of_SPEG}.

\begin{algorithm}[htbp]
  \caption{\textsc{Specular gradient (SPEG) method}}
  \label{alg:SPEG}
  \begin{algorithmic}[1]
    \REQUIRE{Objective function $f: \mathbb{R}^n \to \mathbb{R}$, initial point $\mathbf{x}_0 \in \mathbb{R}^n$, step size $\left\{ h_k \right\}_{k=0}^K \subset (0, \infty)$, tolerance $\eta \in (0, \infty)$, maximum iterations $K \in \mathbb{N}$}
    \ENSURE{Approximate minimizer $\mathbf{x}^{\ast}$}
    \STATE{$k \gets 0$}
    \STATE{$\mathbf{x} \gets \mathbf{x}_0$}
    \STATE{$\mathbf{x}^{\ast} \gets \mathbf{x}_0$}
    \WHILE{$k < K$}
      \STATE{$\mathbf{g} \gets \sg f(\mathbf{x})$}
      \IF{$\left\| \mathbf{g} \right\|_{\mathbb{R}^n} < \eta$}
        \BREAK
      \ENDIF
      \STATE{$\mathbf{x} \gets \mathbf{x} - h_k \, \mathbf{g}$}
      \IF{$f(\mathbf{x}) < f(\mathbf{x}^{\ast})$}
        \STATE{$\mathbf{x}^{\ast} \gets \mathbf{x}$}
      \ENDIF
      \STATE{$k \gets k + 1$}
    \ENDWHILE
    \RETURN $\mathbf{x}^{\ast}$
  \end{algorithmic}
\end{algorithm}

To improve the computational efficiency, we adopt the stochastic approach used in the stochastic gradient descent (SGD) method, which involves the computation of the gradient of a randomly selected component function rather than the gradient of the entire objective function.
In fact, the Elastic Net objective \eqref{ex:obj} can be expressed as a sum of component functions.

Consider the unconstrained optimization problem 
\begin{equation}    \label{OPT:unconstrained_convex_opt_finite_sum}
    f(\mathbf{x}) = \dfrac{1}{m} \sum_{j=1}^m f_j(\mathbf{x}),
\end{equation}
where $f_j : \mathbb{R}^n \to \mathbb{R}$ is convex and specularly Fr\'echet differentiable in $\mathbb{R}^n$.

\begin{definition}
    Consider the unconstrained optimization problem \eqref{OPT:unconstrained_convex_opt_finite_sum}.
    The \emph{stochastic specular gradient} (S-SPEG) method is an algorithm that generates a sequence $\left\{ \mathbf{x}_k \right\}_{k=0}^{\infty}$ according to the formula
    \begin{equation} \label{mtd: stochastic specular gradient method}
        \mathbf{x}_{k+1} = \mathbf{x}_k - h_k \sg f_{\xi_k}(\mathbf{x}_k)
    \end{equation}
    for each $k \in \mathbb{N} \cup \left\{ 0 \right\}$, where $\left\{ h_k \right\}_{k=0}^{\infty} \subset (0, \infty)$ is a sequence of step sizes, $\left\{ \xi_k \right\}_{k=0}^{\infty} \subset \mathbb{N}$ is a sequence of independent random variables with $1 \leq \xi_k \leq m$, and $\mathbf{x}_0 \in \mathbb{R}^n$ is a starting point. 
    Here, we assume that, for each $k \in \mathbb{N} \cup \left\{ 0 \right\}$, the equality
    \begin{equation}    \label{eq:S-SPEG-1}
        \mathbb{E}\left[ \sg f_{\xi_k}(\mathbf{x}_k) \mid \, \mathbf{x}_k \right] = \sg f(\mathbf{x}_k)
    \end{equation}
    holds whenever $\left\| \sg f_{\xi_k}(\mathbf{x}_k) \right\|_{\mathbb{R}^n} \neq 0$.
\end{definition}

Here, we provide the pseudocode of S-SPEG as in \cref{alg:S-SPEG}.
This algorithm minimizes an objective function $f$ as in the unconstrained optimization problem \eqref{OPT:unconstrained_convex_opt_finite_sum}. 

\begin{algorithm}[htbp]
  \caption{\textsc{Stochastic specular gradient (S-SPEG) method}}
  \label{alg:S-SPEG}
  \begin{algorithmic}[1]
    \REQUIRE{Component functions $f_j : \mathbb{R}^n \to \mathbb{R}$ for $j = 1, 2, \ldots, m$, initial point $\mathbf{x}_0 \in \mathbb{R}^n$, step size $\left\{ h_k \right\}_{k=0}^K \subset (0, \infty)$, tolerance $\eta \in (0, \infty)$, maximum iterations $K \in \mathbb{N}$}
    \ENSURE{Approximate minimizer $\mathbf{x}^{\ast}$}
    \STATE{$k \gets 0$}
    \STATE{$\mathbf{x} \gets \mathbf{x}_0$}
    \STATE{$\mathbf{x}^{\ast} \gets \mathbf{x}_0$}
    \WHILE{$k < K$}
      \STATE{$j \gets \text{Random}(1, m)$}
      \STATE{$\mathbf{g} \gets \sg f_j(\mathbf{x})$}
      \IF{$\left\| \mathbf{g} \right\|_{\mathbb{R}^n} < \eta$}
        \BREAK
      \ENDIF
      \STATE{$\mathbf{x} \gets \mathbf{x} - h_k \, \mathbf{g}$}
      \IF{$f(\mathbf{x}) < f(\mathbf{x}^{\ast})$}
        \STATE{$\mathbf{x}^{\ast} \gets \mathbf{x}$}
      \ENDIF
      \STATE{$k \gets k + 1$}
    \ENDWHILE
    \RETURN $\mathbf{x}^{\ast}$
  \end{algorithmic}
\end{algorithm}

We employ one further idea.
In the first experiment, SPEG minimizes the objective function more rapidly than the other methods in the early iterations (for instance, $k \leq 5$).
This motivates us to consider a hybrid method that initially uses SPEG and then switches to another method. In this paper, we specifically switch to S-SPEG.

\begin{definition}
  We define the \emph{hybrid specular gradient} (H-SPEG) method with switching iteration $k \in \mathbb{N}$ as a method that employs the specular gradient method for the first $k$ iterations and switches to the stochastic specular gradient method for the remaining iterations.    
\end{definition}

\section{Convergence results}   \label{sec:convergence}

We prove the convergence of the specular gradient method and the stochastic specular gradient method.
As in the subgradient method, basic inequalities can be obtained for the specular gradient method.

\begin{lemma}   \label{lem:basic_ineq_for_SPEG}
    Consider the problem \eqref{OPT:unconstrained_convex_opt}, where $f : \mathbb{R}^n \to \mathbb{R}$ is convex and specularly Fr\'echet differentiable in $\mathbb{R}^n$.
    Let $\left\{ \mathbf{x}_k \right\}_{k=0}^{\infty}$ be the sequence generated by the specular gradient method \eqref{mtd: specular gradient method}.
    If $\mathbf{x}^{\ast} \in \mathbb{R}^n$ is a minimizer of $f$, then the sequence generated by the specular gradient method \eqref{mtd: specular gradient method} satisfies the following estimate:
    \begin{equation}   \label{ineq:basic_ineq_for_SPEG-1}
        f(\mathbf{x}_k^{\diamond}) - f(\mathbf{x}^{\ast}) \leq \dfrac{\left\| \mathbf{x}_0 - \mathbf{x}^{\ast} \right\|_{\mathbb{R}^n}^2 + \sum_{\ell = 0}^k h_{\ell}^2 \left\| \sg f(\mathbf{x}_{\ell}) \right\|_{\mathbb{R}^n}^2}{2 \sum_{\ell=0}^k h_{\ell}}
    \end{equation}
    for each $k \in \mathbb{N} \cup \left\{ 0 \right\}$.
\end{lemma}

\begin{proof}
    By applying \cref{thm:sg_is_subd}, we obtain, for each $\ell \in \mathbb{N} \cup \left\{ 0 \right\}$,
    \begin{align}
        \left\| \mathbf{x}_{\ell+1} - \mathbf{x}^{\ast} \right\|_{\mathbb{R}^n}^2
        &= \left\| \mathbf{x}_{\ell} - h_{\ell} \sg f(\mathbf{x}_{\ell}) - \mathbf{x}^{\ast} \right\|_{\mathbb{R}^n}^2 \label{ineq:basic_ineq_for_SPEG-2}  \\
        &= \left\| \mathbf{x}_{\ell} - \mathbf{x}^{\ast} \right\|_{\mathbb{R}^n}^2 - 2 h_{\ell} \sg f(\mathbf{x}_{\ell}) \innerprd (\mathbf{x}_{\ell} - \mathbf{x}^{\ast}) + h_{\ell}^2 \left\| \sg f(\mathbf{x}_{\ell}) \right\|_{\mathbb{R}^n}^2  \nonumber \\
        &\leq \left\| \mathbf{x}_{\ell} - \mathbf{x}^{\ast} \right\|_{\mathbb{R}^n}^2 - 2 h_{\ell} \left[ f(\mathbf{x}_{\ell}) - f(\mathbf{x}^{\ast}) \right] + h_{\ell}^2 \left\| \sg f(\mathbf{x}_{\ell}) \right\|_{\mathbb{R}^n}^2, \label{ineq:basic_ineq_for_SPEG-3}
    \end{align}
    namely 
    \begin{equation}    \label{ineq:basic_ineq_for_SPEG-4}
        2 h_{\ell} \left[ f(\mathbf{x}_{\ell}) - f(\mathbf{x}^{\ast}) \right] \leq \left\| \mathbf{x}_{\ell} - \mathbf{x}^{\ast} \right\|_{\mathbb{R}^n}^2 - \left\| \mathbf{x}_{\ell+1} - \mathbf{x}^{\ast} \right\|_{\mathbb{R}^n}^2 + h_{\ell}^2 \left\| \sg f(\mathbf{x}_{\ell}) \right\|_{\mathbb{R}^n}^2.
    \end{equation}    
    Summing the inequalities \eqref{ineq:basic_ineq_for_SPEG-4} from $\ell = 0$ to $\ell = k$ yields that 
    \begin{align}
        2 \sum_{\ell=0}^k h_{\ell} \left[ f(\mathbf{x}_{\ell}) - f(\mathbf{x}^{\ast}) \right] 
        &\leq \left\| \mathbf{x}_0 - \mathbf{x}^{\ast} \right\|_{\mathbb{R}^n}^2 - \left\| \mathbf{x}_{k+1} - \mathbf{x}^{\ast} \right\|_{\mathbb{R}^n}^2 + \sum_{\ell = 0}^k h_{\ell}^2 \left\| \sg f(\mathbf{x}_{\ell}) \right\|_{\mathbb{R}^n}^2  \nonumber  \\
        &\leq \left\| \mathbf{x}_0 - \mathbf{x}^{\ast} \right\|_{\mathbb{R}^n}^2 + \sum_{\ell = 0}^k h_{\ell}^2 \left\| \sg f(\mathbf{x}_{\ell}) \right\|_{\mathbb{R}^n}^2.     \label{ineq:basic_ineq_for_SPEG-5}
    \end{align}
    From the definition \eqref{def: best iteration}, we find that 
    \begin{equation}    \label{ineq:basic_ineq_for_SPEG-6}
        \sum_{\ell=0}^k h_{\ell} \left[ f(\mathbf{x}_{\ell}) - f(\mathbf{x}^{\ast}) \right] 
        \geq \left[ f(\mathbf{x}_k^{\diamond}) - f(\mathbf{x}^{\ast}) \right] \sum_{\ell=0}^k h_{\ell}.
    \end{equation}
    The inequality \eqref{ineq:basic_ineq_for_SPEG-1} follows from the combination of \eqref{ineq:basic_ineq_for_SPEG-5} and \eqref{ineq:basic_ineq_for_SPEG-6}. 
\end{proof}

The inequality \eqref{ineq:basic_ineq_for_SPEG-1} suggests that, similar to the subgradient method, various step sizes can be considered. However, this paper provides a proof only for one of these possible choices. 
More precisely, we prove the convergence of the specular gradient method using square-summable but not summable step sizes.

\begin{theorem} \label{thm: convergence_of_SPEG}
    Consider the unconstrained optimization problem \eqref{OPT:unconstrained_convex_opt}.
    Let $f : \mathbb{R}^n \to \mathbb{R}$ be convex and specularly Fr\'echet differentiable in $\mathbb{R}^n$.
    Let $\left\{ \mathbf{x}_k \right\}_{k=0}^{\infty}$ be the sequence generated by the specular gradient method \eqref{mtd: specular gradient method} with the step size
    \begin{displaymath}
            h_k =
            \begin{cases}
                \dfrac{t_k}{\left\| \sg f(\mathbf{x}_k) \right\|_{\mathbb{R}^n}} > 0    &    \mbox{if } \left\| \sg f(\mathbf{x}_k) \right\|_{\mathbb{R}^n} \neq 0,    \\
                0    &    \mbox{if } \left\| \sg f(\mathbf{x}_k) \right\|_{\mathbb{R}^n} = 0    ,  \\
            \end{cases}
    \end{displaymath}
    where $t_k > 0$ and 
    \begin{enumerate}[label=\rm{(H\arabic*)}, ref=\rm{(H\arabic*)}, start=4, itemsep=1ex]
        \item $\displaystyle \sum_{k=0}^{\infty} t_k = \infty$, \label{H4}
        \item $\displaystyle \sum_{k=0}^{\infty} t_k^2 < \infty$. \label{H5}
    \end{enumerate}
    Then, the following statements hold.
    \begin{enumerate}[label=\upshape(\alph*), itemsep=1ex]
        \item If $\sg f(\mathbf{x}_{\kappa}) = 0$ for some $\kappa \in \mathbb{N} \cup \left\{ 0 \right\}$, then $\mathbf{x}_{\kappa}$ is a minimizer of $f$.    \label{thm:conv_SPEG-1}
        \item If $\sg f(\mathbf{x}_k) \neq 0$ for all $k \in \mathbb{N} \cup \left\{ 0 \right\}$ and there exists a minimizer of $f$, then $\left\{ \mathbf{x}_k \right\}_{k=0}^{\infty}$ converges to one of the minimizers of $f$.  \label{thm:conv_SPEG-2}
    \end{enumerate}    
\end{theorem}

\begin{proof}
    \Cref{thm:conv_SPEG-1} follows from \cref{cor:sg_is_subd-4} of \cref{cor:sg_is_subd}.

    Now, we show \cref{thm:conv_SPEG-2}.
    Let $\mathbf{x}^{\ast}$ be a minimizer of $f$.
    From the same calculations as in \cref{lem:basic_ineq_for_SPEG}, we obtain the following two inequalities.
    First, from the inequality \eqref{ineq:basic_ineq_for_SPEG-3}, we have 
    \begin{equation}    \label{ineq:conv_SPEG-1}
        \left\| \mathbf{x}_{\ell+1} - \mathbf{x}^{\ast} \right\|_{\mathbb{R}^n}^2
        \leq \left\| \mathbf{x}_{\ell} - \mathbf{x}^{\ast} \right\|_{\mathbb{R}^n}^2 - 2 h_{\ell} \left[ f(\mathbf{x}_{\ell}) - f(\mathbf{x}^{\ast}) \right] + t_{\ell}^2,
    \end{equation}
    for each $\ell \in \mathbb{N} \cup \left\{ 0 \right\}$.
    Second, from the inequality \eqref{ineq:basic_ineq_for_SPEG-5}, we have 
    \begin{equation}    \label{ineq:conv_SPEG-2}
        2 \sum_{\ell=0}^k h_{\ell} \left[ f(\mathbf{x}_{\ell}) - f(\mathbf{x}^{\ast}) \right] 
        \leq \left\| \mathbf{x}_0 - \mathbf{x}^{\ast} \right\|_{\mathbb{R}^n}^2 + \sum_{\ell=0}^k t_{\ell}^2
    \end{equation}
    for each $k \in \mathbb{N} \cup \left\{ 0 \right\}$.

    We claim that 
    \begin{equation}    \label{ineq:conv_SPEG-3}
        \liminf_{k \to \infty} f(\mathbf{x}_k) = f(\mathbf{x}^{\ast}).
    \end{equation}
    Suppose to the contrary that 
    \begin{displaymath}
        f(\mathbf{x}^{\ast}) < \liminf_{k \to \infty} f(\mathbf{x}_k) = \sup_{k \in \mathbb{N} \cup \left\{ 0 \right\}} \left( \inf_{\ell \geq k} f(\mathbf{x}_{\ell}) \right).
    \end{displaymath}
    Then, there exist $L \in \mathbb{N} \cup \left\{ 0 \right\}$ and $\varepsilon_0 > 0$ such that 
    \begin{displaymath}
        \inf_{\ell \geq L} f(\mathbf{x}_{\ell}) > f(\mathbf{x}^{\ast}) + \varepsilon_0,
    \end{displaymath}
    namely 
    \begin{equation}    \label{ineq:conv_SPEG-4}
        f(\mathbf{x}_{\ell}) -  f(\mathbf{x}^{\ast})\geq \varepsilon_0 
    \end{equation}
    for all $\ell \geq L$.
    From \eqref{ineq:conv_SPEG-2}, \labelcref{H5}, \eqref{ineq:conv_SPEG-4} and \labelcref{H4}, we find that 
    \begin{align*}
        \infty 
        &> \sum_{\ell=0}^{\infty} h_{\ell} [f(\mathbf{x}_{\ell}) - f(\mathbf{x}^{\ast})] \\
        &= \sum_{\ell=0}^{L-1} h_{\ell} [f(\mathbf{x}_{\ell}) - f(\mathbf{x}^{\ast})] + \sum_{\ell=L}^{\infty} h_{\ell} [f(\mathbf{x}_{\ell}) - f(\mathbf{x}^{\ast})]  \\
        &\geq \sum_{\ell=0}^{L-1} h_{\ell} [f(\mathbf{x}_{\ell}) - f(\mathbf{x}^{\ast})] + \varepsilon_0 \sum_{\ell=L}^{\infty} h_{\ell} = \infty,
    \end{align*}
    a contradiction. 
    Therefore, \eqref{ineq:conv_SPEG-3} holds.
    
    Then, there exists a subsequence $\left\{ \mathbf{x}_{k_j} \right\}_{j=0}^{\infty} \subset \left\{ \mathbf{x}_k \right\}_{k = 0}^{\infty}$ such that 
    \begin{equation} \label{ineq:conv_SPEG-5}
        \lim_{j \to \infty} f(\mathbf{x}_{k_j}) = f(\mathbf{x}^{\ast}).
    \end{equation}
    By the optimality of $\mathbf{x}^{\ast}$, the inequality \eqref{ineq:conv_SPEG-1} can be reduced to 
    \begin{equation}    \label{ineq:conv_SPEG-6}
        \left\| \mathbf{x}_{\ell+1} - \mathbf{x}^{\ast} \right\|_{\mathbb{R}^n}^2 \leq \left\| \mathbf{x}_{\ell} - \mathbf{x}^{\ast} \right\|_{\mathbb{R}^n}^2 + t_{\ell}^2.
    \end{equation}
    Iterating the inequalities \eqref{ineq:conv_SPEG-6} from $\ell = 0$ to $\ell = k - 1$ yields that 
    \begin{displaymath}
        0 
        \leq \left\| \mathbf{x}_k - \mathbf{x}^{\ast} \right\|_{\mathbb{R}^n}^2 
        \leq \left\| \mathbf{x}_0 - \mathbf{x}^{\ast} \right\|_{\mathbb{R}^n}^2 + \sum_{\ell = 0}^{k-1} t_{\ell}^2  
        \leq \left\| \mathbf{x}_0 - \mathbf{x}^{\ast} \right\|_{\mathbb{R}^n}^2 + \sum_{\ell = 0}^{\infty} t_{\ell}^2
        < \infty .
    \end{displaymath}
    The finiteness follows from \labelcref{H5}.
    Therefore, the sequence $\left\{ \mathbf{x}_k \right\}_{k=0}^{\infty}$ is bounded, and hence the subsequence $\left\{ \mathbf{x}_{k_j} \right\}_{j = 0}^{\infty}$ is also bounded. 
    By the Bolzano-Weierstrass theorem, there exist a subsequence $\left\{ \mathbf{x}_{k_{j_m}} \right\}_{m = 0}^{\infty} \subset \left\{ \mathbf{x}_{k_j} \right\}_{j=0}^{\infty}$ and a point $\overline{\mathbf{x}} \in \mathbb{R}^n$ such that
    \begin{equation}    \label{ineq:conv_SPEG-7}
        \lim_{m \to \infty} \mathbf{x}_{k_{j_m}} = \overline{\mathbf{x}}.
    \end{equation}

    From the continuity of $f$, we see that 
    \begin{displaymath}
        f(\overline{\mathbf{x}}) = \lim_{m \to \infty} f(\mathbf{x}_{k_{j_m}}) = \lim_{j \to \infty} f(\mathbf{x}_{k_j}) = f(\mathbf{x}^{\ast})
    \end{displaymath}
    by \eqref{ineq:conv_SPEG-5} and \eqref{ineq:conv_SPEG-7}.
    This means that $\overline{\mathbf{x}}$ is a minimizer of $f$. 

    Finally, we claim that $\mathbf{x}_k$ converges to $\overline{\mathbf{x}}$ as $k \to \infty$.
    Let $\varepsilon > 0$ be arbitrary.
    From the equality \eqref{ineq:conv_SPEG-7} and \labelcref{H5}, there exists $M \in \mathbb{N} \cup \left\{ 0 \right\}$ such that
    \begin{displaymath}
        \left\| \mathbf{x}_K - \overline{\mathbf{x}} \right\|_{\mathbb{R}^n}^2 < \dfrac{\varepsilon}{2}
        \qquad\text{and}\qquad
        \sum_{\ell = K}^{\infty} t_{\ell}^2 < \dfrac{\varepsilon}{2},
    \end{displaymath}
    where $K := k_{j_{M}}$.
    If $k > K$, then iterating the inequality \eqref{ineq:conv_SPEG-6} with $\overline{\mathbf{x}}$ in place of $\mathbf{x}^{\ast}$ from $\ell = K$ to $\ell = k - 1$ implies that
    \begin{displaymath}
        \left\| \mathbf{x}_k - \overline{\mathbf{x}} \right\|_{\mathbb{R}^n}^2 
        \leq \left\| \mathbf{x}_K - \overline{\mathbf{x}} \right\|_{\mathbb{R}^n}^2 + \sum_{\ell = K}^{k-1} t_{\ell}^2 
        \leq \left\| \mathbf{x}_K - \overline{\mathbf{x}} \right\|_{\mathbb{R}^n}^2 + \sum_{\ell = K}^{\infty} t_{\ell}^2 
        < \varepsilon.
    \end{displaymath}
    This implies that $\mathbf{x}_k \to \overline{\mathbf{x}}$ as $k \to \infty$.
\end{proof}

Note that the specular method, much like the standard subgradient method, does not guarantee the convergence of $\left\| \sg f(\mathbf{x}_k) \right\|_{\mathbb{R}^n}$ to zero as $k \to \infty$.

\begin{remark}
    Consider the constrained optimization problem 
    \begin{displaymath} 
        \min_{\mathbf{x} \in E} f(\mathbf{x}),
    \end{displaymath}
    where $E \subset \mathbb{R}^n$ is a convex closed set, and the function $f : \mathbb{R}^n \to \mathbb{R}$ is convex and specularly Fr\'echet differentiable in $\mathbb{R}^n$.

    We modify the specular gradient method \eqref{mtd: specular gradient method} to be 
    \begin{equation}    \label{mtd: projected specular gradient method}
        \mathbf{x}_{k+1} = \pi_E \left( \mathbf{x}_k - h_k \sg f(\mathbf{x}_k) \right),
    \end{equation}
    for $k \in \mathbb{N} \cup \left\{ 0 \right\}$, where $\pi_E(\mathbf{x})$ is the orthogonal projection of a point $\mathbf{x} \in \mathbb{R}^n$ onto the set $E$.

    Since the projection operator is nonexpansive, we find that 
    \begin{align*}
        \left\| \mathbf{x}_{k+1} - \mathbf{x}^{\ast} \right\|_{\mathbb{R}^n}^2 
        &= \left\| \pi_E(\mathbf{x}_k - h_k \sg f(\mathbf{x}_k)) - \pi_E (\mathbf{x}^{\ast}) \right\|_{\mathbb{R}^n}^2   \\
        &\leq \left\| \mathbf{x}_k - h_k \sg f(\mathbf{x}_k) - \mathbf{x}^{\ast} \right\|_{\mathbb{R}^n}^2 .
    \end{align*}
    Modifying the proof of \cref{lem:basic_ineq_for_SPEG} from the inequality \eqref{ineq:basic_ineq_for_SPEG-2}, one can prove \cref{thm: convergence_of_SPEG} for the modified method \eqref{mtd: projected specular gradient method}.
\end{remark}

Finally, we provide the basic convergence result for the stochastic specular gradient method using square-summable but not summable step sizes. 

\begin{theorem} \label{thm: convergence_of_S_SPEG}
    Consider the unconstrained optimization problem \eqref{OPT:unconstrained_convex_opt} 
    where the objective function $f$ is given by \eqref{OPT:unconstrained_convex_opt_finite_sum}.
    Let $\left\{ \mathbf{x}_k \right\}_{k=0}^{\infty}$ be the sequence generated by the stochastic specular gradient method \eqref{mtd: stochastic specular gradient method} with the step size
    \begin{equation}    \label{eq: stochastic square summable but not summable step size}
        h_k =
        \begin{cases}
            \dfrac{t_k}{\left\| \sg f_{\xi_k}(\mathbf{x}_k) \right\|_{\mathbb{R}^n}} > 0,    &    \mbox{if } \left\| \sg f_{\xi_k}(\mathbf{x}_k) \right\|_{\mathbb{R}^n} \neq 0,    \\
            0    &    \mbox{if } \left\| \sg f_{\xi_k}(\mathbf{x}_k) \right\|_{\mathbb{R}^n} = 0 ,     \\
        \end{cases}
    \end{equation}
    where $t_k > 0$ with \labelcref{H4} and \labelcref{H5}.
    Furthermore, assume that 
    \begin{enumerate}[label=\rm{(H\arabic*)}, ref=\rm{(H\arabic*)}, start=6, itemsep=1ex]
        \item there exists $C > 0$ such that $\left\| \sg f_j (\mathbf{x}_k) \right\|_{\mathbb{R}^n} \leq C$ for all $k \in \mathbb{N} \cup \left\{ 0 \right\}$ and $1 \leq j \leq m$.
    \end{enumerate}
    Then, the following statements hold.
    \begin{enumerate}[label=\upshape(\alph*), itemsep=1ex]
        \item If $\left\| \sg f_{\xi_{\kappa}}(\mathbf{x}_{\kappa}) \right\|_{\mathbb{R}^n} = 0$ for some $\kappa \in \mathbb{N} \cup \left\{ 0 \right\}$, then $\mathbf{x}_{\kappa}$ is a minimizer of $f_{\xi_{\kappa}}$.
        \item If $\left\| \sg f_{\xi_k}(\mathbf{x}_k) \right\|_{\mathbb{R}^n} > 0$ for all $k \in \mathbb{N} \cup \left\{ 0 \right\}$ and there exists a minimizer of $f$, then $\left\{ f(\mathbf{x}_k^\diamond) \right\}_{k=0}^{\infty}$ converges to $f(\mathbf{x}^{\ast})$ in expectation and in probability, that is, 
    \begin{displaymath}
        \lim_{k \to \infty} \mathbb{E} \left[ f(\mathbf{x}_k^{\diamond}) \right] = f(\mathbf{x}^{\ast})
    \end{displaymath}
    and 
    \begin{displaymath}
        \lim_{k \to \infty} \mathbb{P} \left[ f(\mathbf{x}_k^{\diamond}) - f(\mathbf{x}^{\ast}) > \varepsilon  \right] = 0
    \end{displaymath}
    for all $\varepsilon > 0$, respectively, where $\mathbf{x}^{\ast}$ is a minimizer of $f$.
    \end{enumerate}
\end{theorem}

\begin{proof}
    For each $\ell \in \mathbb{N} \cup \left\{ 0 \right\}$, we find that 
    \begin{align*}
        \mathbb{E} \left[ \left\| \mathbf{x}_{\ell+1} - \mathbf{x}^{\ast} \right\|_{\mathbb{R}^n}^2 | \, \mathbf{x}_{\ell}  \right]
        =&\, \mathbb{E} \left[ \left\| \mathbf{x}_{\ell}  - \mathbf{x}^{\ast} - h_{\ell} \sg f_{\xi_{\ell}}(\mathbf{x}_{\ell}) \right\|_{\mathbb{R}^n}^2 | \, \mathbf{x}_{\ell}  \right]   \\
        =& \left\| \mathbf{x}_{\ell} - \mathbf{x}^{\ast} \right\|_{\mathbb{R}^n}^2 - 2 \mathbb{E} \left[ h_{\ell} \, \sg f_{\xi_{\ell}}(\mathbf{x}_{\ell}) \innerprd (\mathbf{x}_{\ell} - \mathbf{x}^{\ast}) | \, \mathbf{x}_{\ell} \right] \\
        &+ \mathbb{E} \left[ h_{\ell}^2 \left\| \sg f_{\xi_{\ell}}(\mathbf{x}_{\ell}) \right\|_{\mathbb{R}^n}^2 | \, \mathbf{x}_{\ell} \right]    \\
        =& \left\| \mathbf{x}_{\ell} - \mathbf{x}^{\ast} \right\|_{\mathbb{R}^n}^2 - 2 t_{\ell} \, \mathbb{E} \left[ \sg f_{\xi_{\ell}}(\mathbf{x}_{\ell}) | \, \mathbf{x}_{\ell} \right] \innerprd (\mathbf{x}_{\ell} - \mathbf{x}^{\ast}) + t_{\ell}^2  \\
        =& \left\| \mathbf{x}_{\ell} - \mathbf{x}^{\ast} \right\|_{\mathbb{R}^n}^2 - 2 t_{\ell} \sg f(\mathbf{x}_{\ell}) \innerprd (\mathbf{x}_{\ell} - \mathbf{x}^{\ast}) + t_{\ell}^2  \\
        \leq& \left\| \mathbf{x}_{\ell} - \mathbf{x}^{\ast} \right\|_{\mathbb{R}^n}^2 - 2 t_{\ell} \left( f(\mathbf{x}_{\ell}) - f(\mathbf{x}^{\ast}) \right)  + t_{\ell}^2
    \end{align*}
    by applying the equality \eqref{eq:S-SPEG-1} and \cref{thm:sg_is_subd}.
    This inequality implies that 
    \begin{displaymath}
        \frac{2 t_{\ell}}{C} \left( f(\mathbf{x}_{\ell}) - f(\mathbf{x}^{\ast}) \right)
        \leq \left\| \mathbf{x}_{\ell} - \mathbf{x}^{\ast} \right\|_{\mathbb{R}^n}^2 - \mathbb{E} \left[ \left\| \mathbf{x}_{\ell+1} - \mathbf{x}^{\ast} \right\|_{\mathbb{R}^n}^2 | \, \mathbf{x}_{\ell}  \right] + t_{\ell}^2
    \end{displaymath}
    by the equality \eqref{eq: stochastic square summable but not summable step size}.
    Taking the expectation, we obtain that 
    \begin{displaymath}
        \frac{2 t_{\ell}}{C} \left( \mathbb{E} \left[ f(\mathbf{x}_{\ell})  \right] - f(\mathbf{x}^{\ast}) \right) 
        \leq \mathbb{E}\left[ \left\| \mathbf{x}_{\ell} - \mathbf{x}^{\ast} \right\|_{\mathbb{R}^n}^2  \right] - \mathbb{E} \left[ \left\| \mathbf{x}_{\ell+1} - \mathbf{x}^{\ast} \right\|_{\mathbb{R}^n}^2 \right]  + t_{\ell}^2.
    \end{displaymath}
    Summing the inequalities from $\ell = 0$ to $\ell = k$ implies that 
    \begin{align*}
        \frac{2}{C} \sum_{\ell = 0}^k t_{\ell} \left( \mathbb{E} \left[ f(\mathbf{x}_{\ell})  \right] - f(\mathbf{x}^{\ast}) \right) 
        &\leq \mathbb{E}\left[ \left\| \mathbf{x}_0 - \mathbf{x}^{\ast} \right\|_{\mathbb{R}^n}^2  \right] - \mathbb{E} \left[ \left\| \mathbf{x}_{k+1} - \mathbf{x}^{\ast} \right\|_{\mathbb{R}^n}^2 \right] + \sum_{\ell = 0}^k t_{\ell}^2    \\
        &\leq C \cdot \mathbb{E}\left[ \left\| \mathbf{x}_0 - \mathbf{x}^{\ast} \right\|_{\mathbb{R}^n}^2  \right] + \sum_{\ell = 0}^k t_{\ell}^2.
    \end{align*}
    Using Jensen's inequality and the above inequality, we see that 
    \begin{align*}
        \mathbb{E} \left[ f(\mathbf{x}_k^{\diamond}) \right] - f(\mathbf{x}^{\ast})
        &\leq \min_{0 \leq \ell \leq k} \mathbb{E} \left[ f(\mathbf{x}_{\ell}) \right] - f(\mathbf{x}^{\ast})     \\
        &\leq \frac{\sum_{\ell = 0}^k t_{\ell} \left( \mathbb{E} \left[ f(\mathbf{x}_{\ell})  \right] - f(\mathbf{x}^{\ast}) \right)}{\sum_{l = 0}^k t_{l}}     \\
        &\leq \dfrac{\mathbb{E}\left[ \left\| \mathbf{x}_0 - \mathbf{x}^{\ast} \right\|_{\mathbb{R}^n}^2  \right] + \sum_{\ell = 0}^k t_{\ell}^2}{2 \, \sum_{\ell = 0}^k t_{\ell}}, 
    \end{align*}
    which converges to zero as $k \to \infty$ by the assumptions \labelcref{H4,H5}.
    Therefore, the convergence in expectation follows. 

    Next, for each $\varepsilon > 0$, Markov's inequality implies that 
    \begin{displaymath}
        \mathbb{P}\left[ f(\mathbf{x}_k^{\diamond}) - f(\mathbf{x}^{\ast}) \geq \varepsilon \right] \leq \dfrac{\mathbb{E} \left[ f(\mathbf{x}_{k}^{\diamond}) - f(\mathbf{x}^{\ast}) \right]}{\varepsilon},
    \end{displaymath}
    which converges to zero as $k \to \infty$ by the convergence in expectation.
    Hence, the convergence in probability follows.
\end{proof}

\section{Numerical results} \label{sec:numerical_results}

We present numerical examples.
In the following experiments, we compare the proposed methods with three standard optimization algorithms: Adaptive Moment Estimation (Adam), Broyden--Fletcher--Goldfarb--Shanno (BFGS), and gradient descent (GD).
Adam is an adaptive stochastic gradient method; see \cite{2017_Kingma}.
The BFGS algorithm is a quasi-Newton method introduced by Broyden \cite{1970_Broyden}, Fletcher \cite{1970_Fletcher}, Goldfarb \cite{1970_Goldfarb}, and Shanno \cite{1970_Shanno}.
A more modern exposition can be found in \cite[Chap.~6]{2006_Nocedal_BOOK}.
Although comparing first-order methods with a quasi-Newton method may not be entirely fair, BFGS serves as a benchmark for evaluating the performance of the first-order methods.

All experiments were conducted on a system equipped with an Intel Core i9-14900K CPU and 32 GB of RAM, without a discrete GPU.
The software environment consisted of Python 3.11, PyTorch (v2.7.0, CPU build, \cite{2024_PyTorch}), SciPy (v1.15.3, \cite{2020_SciPy}), and \texttt{specular-differentiation} (v1.2.0, \cite{2026s_Jung_SIAM}).

Consider the Elastic Net objective function $f : \mathbb{R}^n \to \mathbb{R}$ given by 
\begin{equation}    \label{ex:obj}
    f(\mathbf{x}) = \dfrac{1}{2m}\left\| A\mathbf{x} - \mathbf{b} \right\|_{\mathbb{R}^m}^2 + \dfrac{\lambda_2}{2} \left\| \mathbf{x} \right\|_{\mathbb{R}^n}^2 + \lambda_1 \left\| \mathbf{x} \right\|_{\ell^1},
\end{equation}
for $\mathbf{x} \in \mathbb{M}^{n \times 1}$, where $A \in \mathbb{M}^{m \times n}$, $\mathbf{b} \in \mathbb{M}^{m \times 1}$, and the regularization parameters $\lambda_1, \lambda_2 \geq 0$ are given. 
This form of regularization was first introduced in \cite{2005_Zou}.

For each experiment, we conducted $20$ trials to minimize the objective function \eqref{ex:obj}.
For each trial, the matrix $A$, the vector $\mathbf{b}$, and the starting point $\mathbf{x}_0$ were drawn independently from the standard normal distribution $\mathcal{N}(0, 1)$.
Since the entries of $A$ are drawn from a continuous distribution, $A$ has full rank almost surely, i.e., $\operatorname{rank}(A)=\min \{m, n\}$.
Then, we computed the mean and the standard deviation of the $20$ values of the objective function found by an optimization method for each iteration.

For each method, the same step size (or learning rate) was used consistently across all subsequent experiments.
The step sizes for GD and Adam were chosen empirically as $0.001$ and $0.01$, respectively.
For SPEG, we consider the step size 
\begin{equation}    \label{eq: step size for examples}
    h_k = \dfrac{4}{(k + 1) \left\| \sg f( \mathbf{x}_k ) \right\|_{\mathbb{R}^n}}
\end{equation}
for each $k \in \mathbb{N} \cup \left\{ 0 \right\}$.

Optimization results for the objective function \eqref{ex:obj} with $m=50$, $n=100$, $\lambda_1 = 0.01$, and $\lambda_2 = 1.0$ are presented in \cref{fig:opt-50-100-0.01-1.0} and \cref{tbl:opt-50-100-0.01-1.0}.
In \cref{fig:opt-50-100-0.01-1.0} and the subsequent figures, the solid lines represent the mean computed over all trials, and the dashed lines indicate the median computed over all trials.

In \cref{fig:opt-50-100-0.01-1.0} and \cref{tbl:opt-50-100-0.01-1.0}, all methods converge except for S-SPEG and H-SPEG.
Although SPEG converges faster than GD and Adam in terms of iteration count, it converges slower than BFGS, and its actual computation time is longer than that of the classical methods.
However, 
This might raise a question regarding the practical utility of SPEG. 
To address this, we present two key justifications. 
First, we demonstrate examples where the classical methods fail to converge, whereas SPEG converges successfully. 
Second, we introduce two strategies to enhance the computational efficiency of SPEG.

\begin{figure}[htbp]
    \centering
    \includegraphics[width=1\textwidth]{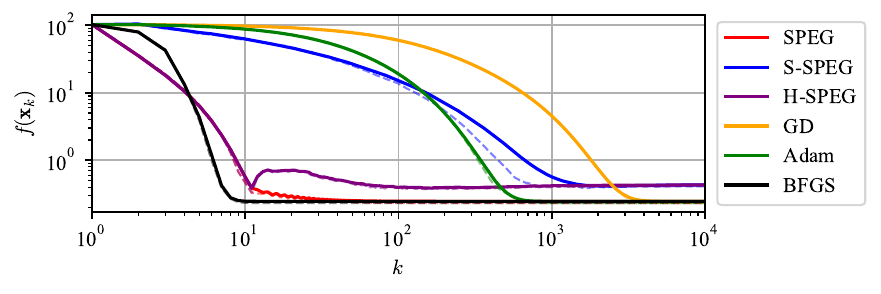}
    \vspace{-1.5\baselineskip}
    \caption{Optimization of the objective function \eqref{ex:obj} with $m=50$, $n=100$, $\lambda_1 = 0.01$, and $\lambda_2 = 1.0$. \label{fig:opt-50-100-0.01-1.0}} 
\end{figure}

\Cref{tbl:opt-50-100-0.01-1.0} and the following tables summarize the statistics of the optimization results.
Specifically, the mean, median, and standard deviation are computed from the $20$ minimum objective function values found by each method.

\begin{table}[htbp]
  \footnotesize
  \renewcommand{\arraystretch}{1.3}
  \caption{Optimization of the objective function \eqref{ex:obj} with $m=50$, $n=100$, $\lambda_1 = 0.01$, and $\lambda_2 = 1.0$.}
  \label{tbl:opt-50-100-0.01-1.0}
  \begin{center}
    \begin{tabular}{|c|c|c|c|} \hline
        Methods & \bf Mean               & \bf Median              & \bf Std.~Dev.           \\ 
      \hline
        SPEG   & $2.4639 \times 10^{-1}$ & $2.3230 \times 10^{-1}$ & $4.6444 \times 10^{-2}$ \\
        S-SPEG & $3.7150 \times 10^{-1}$ & $3.5627 \times 10^{-1}$ & $7.1775 \times 10^{-2}$ \\
        H-SPEG & $3.1154 \times 10^{-1}$ & $3.0583 \times 10^{-1}$ & $7.2139 \times 10^{-2}$ \\
        GD     & $2.4639 \times 10^{-1}$ & $2.3230 \times 10^{-1}$ & $4.6445 \times 10^{-2}$ \\
        Adam   & $2.4640 \times 10^{-1}$ & $2.3230 \times 10^{-1}$ & $4.6442 \times 10^{-2}$ \\
        BFGS   & $2.4639 \times 10^{-1}$ & $2.3230 \times 10^{-1}$ & $4.6445 \times 10^{-2}$ \\
      \hline
    \end{tabular}
  \end{center}
\end{table}

In the following experiments, we fix the switching iteration of H-SPEG at $k=10$.
Also, S-SPEG and H-SPEG use the step size \eqref{eq: step size for examples}.
For optimization of the objective function \eqref{ex:obj} with $m=100$ and $n=100$, \cref{fig:lambda-comparison} compares SPEG, S-SPEG, and H-SPEG for $\lambda_1, \lambda_2 \in \{0.1, 1.0, 10.0\}$.

From \cref{fig:lambda-comparison}, one can infer that the behavior of H-SPEG tends to resemble that of SPEG as $\lambda_1$ and $\lambda_2$ increase.\footnote{We also obtained similar results for $(m=100, n=50)$ and $(m=50, n=100)$, which are omitted here for brevity.}
Although S-SPEG and H-SPEG do not converge faster than SPEG in terms of iteration count, they are faster than SPEG in terms of actual computation time; in some cases, they are even faster than BFGS.
Therefore, S-SPEG and H-SPEG show improved computational efficiency. 
Since S-SPEG and H-SPEG tend to converge rapidly for cases where $\lambda_1$ and $\lambda_2$ are large, we focus on such cases.

\begin{figure}[htbp]
  \centering
  \includegraphics[width=\textwidth]{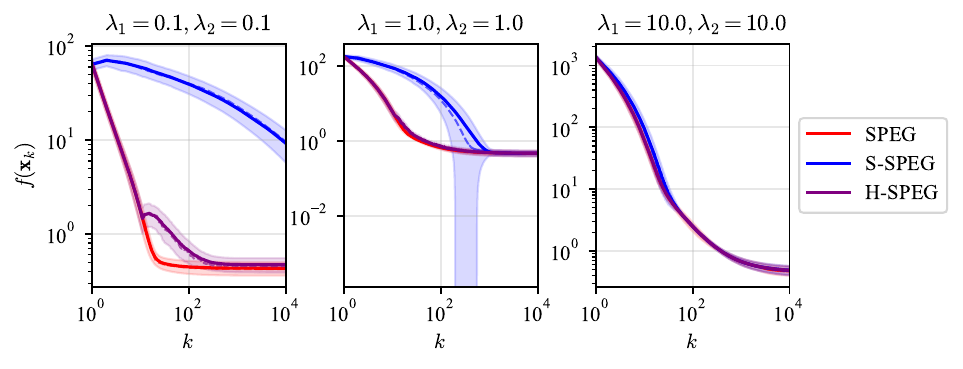}
  \vspace{-1.5\baselineskip}
  \caption{Comparison of SPEG, S-SPEG, and H-SPEG for optimization of the objective function \eqref{ex:obj} with $m=100$ and $n=100$ under varying $\lambda_1$ and $\lambda_2$.}
  \label{fig:lambda-comparison}
\end{figure}

We now discuss the mathematical significance and computational efficiency of the proposed methods.
\Cref{fig:opt-500-100-100.0-1.0} and \cref{tbl:opt-500-100-100.0-1.0} present the optimization results of the objective function \eqref{ex:obj} with $m=500$, $n=100$, $\lambda_1 = 100.0$, and $\lambda_2 = 1.0$.
This example is of particular significance. 
Not only do the proposed methods successfully find minimizers where classical methods fail, but S-SPEG and H-SPEG also exhibit computational speeds comparable to BFGS (with a difference of less than one second).
Observe that the dashed black line for BFGS indicates that its results are not stable; although BFGS is sometimes faster than the proposed methods, it performs poorly in terms of the mean.
Consequently, S-SPEG and H-SPEG serve as excellent alternatives in scenarios where the computation is expensive or where classical methods fail to converge.

\begin{figure}[htbp]
  \centering
  \includegraphics[width=1\textwidth]{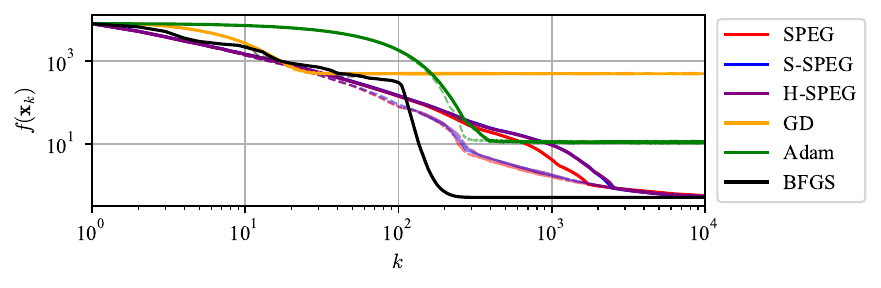}
  \vspace{-1.5\baselineskip}
  \caption{Optimization for the objective function \eqref{ex:obj} with $m=500$, $n=100$, $\lambda_1 = 100.0$, and $\lambda_2 = 1.0$.}
  \label{fig:opt-500-100-100.0-1.0}
\end{figure}

\begin{table}[htbp]
  \footnotesize
  \renewcommand{\arraystretch}{1.3}
  \caption{Optimization for the objective function \eqref{ex:obj} with $m=500$, $n=100$, $\lambda_1 = 100.0$, and $\lambda_2 = 1.0$.}
  \label{tbl:opt-500-100-100.0-1.0}
  \begin{center}
    \begin{tabular}{|c|c|c|c|} \hline
        Methods & \bf Mean                & \bf Median              & \bf Std.~Dev.           \\ 
      \hline
        SPEG    & $5.6041 \times 10^{-1}$ & $5.5599 \times 10^{-1}$ & $3.5467 \times 10^{-2}$ \\
        S-SPEG  & $5.3234 \times 10^{-1}$ & $5.2704 \times 10^{-1}$ & $3.4641 \times 10^{-2}$ \\
        H-SPEG  & $5.3208 \times 10^{-1}$ & $5.3105 \times 10^{-1}$ & $3.6578 \times 10^{-2}$ \\
        GD      & $4.7761 \times 10^{2}$  & $4.8070 \times 10^{2}$  & $1.0166 \times 10^{1}$  \\
        Adam    & $8.2266 \times 10^{0}$  & $8.1931 \times 10^{0}$  & $1.8712 \times 10^{-1}$ \\
        BFGS    & $1.5791 \times 10^{1}$  & $5.0806 \times 10^{-1}$ & $6.8360 \times 10^{1}$  \\
      \hline
    \end{tabular}
  \end{center}
\end{table}

Finally, we present the smooth case where $\lambda_1 = 0$ and $\lambda_2 = 0$, as shown in \cref{fig:opt_500_100_0.0_0.0} and \cref{tbl:opt_500_100_0.0_0.0}.
As expected, BFGS outperforms the other methods not only in terms of iteration count but also in terms of computational efficiency.
In such scenarios, the proposed methods may offer fewer advantages over classical methods in smooth optimization problems.

\begin{figure}[htbp]
    \centering
    \includegraphics[width=1\textwidth]{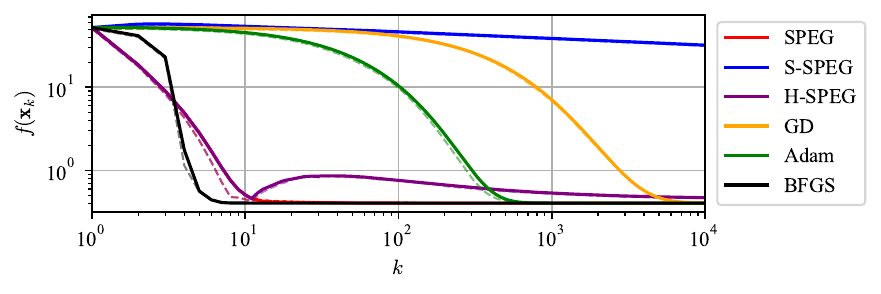}
    \vspace{-1.5\baselineskip}
    \caption{Optimization for the objective function \eqref{ex:obj} with $m=500$, $n=100$, $\lambda_1 = 0.0$, and $\lambda_2 = 0.0$.}
    \label{fig:opt_500_100_0.0_0.0}
\end{figure}

\begin{table}[htbp]
  \footnotesize
  \renewcommand{\arraystretch}{1.3}
  \caption{Optimization for the objective function \eqref{ex:obj} with $m=500$, $n=100$, $\lambda_1 = 0.0$, and $\lambda_2 = 0.0$.}
  \label{tbl:opt_500_100_0.0_0.0}
  \begin{center}
    \begin{tabular}{|c|c|c|c|} \hline
        Methods & \bf Mean                & \bf Median              & \bf Std.~Dev.           \\ 
      \hline
        SPEG    & $4.0475 \times 10^{-1}$ & $4.0525 \times 10^{-1}$ & $3.3922 \times 10^{-2}$ \\
        S-SPEG  & $3.2097 \times 10^{1}$  & $3.1328 \times 10^{1}$  & $6.0106 \times 10^{0}$  \\
        H-SPEG  & $4.4824 \times 10^{-1}$ & $4.1472 \times 10^{-1}$ & $1.3012 \times 10^{-1}$ \\
        GD      & $4.0584 \times 10^{-1}$ & $4.0685 \times 10^{-1}$ & $3.3998 \times 10^{-2}$ \\
        Adam    & $4.0475 \times 10^{-1}$ & $4.0525 \times 10^{-1}$ & $3.3922 \times 10^{-2}$ \\
        BFGS    & $4.0475 \times 10^{-1}$ & $4.0525 \times 10^{-1}$ & $3.3922 \times 10^{-2}$ \\
      \hline
    \end{tabular}
  \end{center}
\end{table}

First, SPEG, S-SPEG, and H-SPEG successfully minimize nonsmooth convex functions for which classical methods fail to converge.
Second, the performance of the proposed methods is highly dependent on the parameters $\lambda_1$ and $\lambda_2$.
Third, the relationship between the number of component functions $m$ and the dimension $n$ significantly influences performance.
Specifically, when $m \leq n$, classical methods tend to converge faster and achieve lower objective values.
However, in the regime where $m > n$, the proposed methods demonstrate superior robustness, often succeeding where classical methods fail.
Fourth, H-SPEG and S-SPEG generally exhibit superior computational efficiency compared to SPEG, whereas SPEG is consistently the slowest.
Finally, further improvements to the proposed methods will require the implementation of adaptive step sizes or the use of second-order specular differentiation.

\bibliographystyle{siamplain}
\bibliography{reference}{}

\end{document}

%% file: shared_info.tex
\usepackage{lipsum}
\usepackage{amsfonts}
\usepackage{amssymb}
\usepackage{bookmark}
\usepackage{bm}
\usepackage{graphicx}
\usepackage{epstopdf}
\usepackage{hyperref}
\usepackage{enumitem}
\usepackage{algorithmic}
\usepackage{booktabs}
\usepackage{multirow}
\usepackage{cleveref}
\usepackage{aliascnt}
\usepackage{tikz-cd}
\usepackage{xspace}
\usepackage[en-US]{datetime2}
\usepackage{orcidlink}

\ifpdf
  \DeclareGraphicsExtensions{.eps,.pdf,.png,.jpg}
\else
  \DeclareGraphicsExtensions{.eps}
\fi

\newsiamremark{remark}{Remark}
\crefname{remark}{Remark}{Remarks}
\Crefname{remark}{Remark}{Remarks}

\newsiamremark{hypothesis}{Hypothesis}
\crefname{hypothesis}{Hypothesis}{Hypotheses}
\newsiamthm{claim}{Claim}
\newsiamremark{fact}{Fact}
\crefname{fact}{Fact}{Facts}

\newaliascnt{example}{theorem}

\aliascntresetthe{example}
\crefname{example}{Example}{Examples}
\Crefname{example}{Example}{Examples}
\crefname{remark}{Remark}{Remarks}
\Crefname{remark}{Remark}{Remarks}

\crefname{enumi}{part}{parts}
\Crefname{enumi}{Part}{Parts}

\newcommand{\sGd}{\widehat{\mkern-2mu d}\mkern1mu}
\newcommand\sd[1][.5]{\mathbin{\vcenter{\hbox{\scalebox{#1}{\,$\bm{\wedge}$}}}}}
\newcommand\innerprd{\mathbin{\vcenter{\hbox{\scalebox{0.5}{$\bullet$}}}}}
\newcommand{\sg}{\mathord{\raisebox{-0.05ex}{\rule{0pt}{1.3ex}\smash{\scalebox{1.37}[1.22]{\ensuremath{\blacktriangledown}}}}\mkern-1.2mu}}

\Crefname{ALC@unique}{Line}{Lines}

\newcommand{\BREAK}{\STATE \textbf{break}}

\headers{Specular gradient methods}{K. Jung}

\title{Specular gradient methods for nonsmooth convex optimization in Euclidean spaces: a subgradient selection strategy\thanks{Date: \today.
\funding{This work is supported in part by Michigan State University and the National Science Foundation Research Traineeship Program (DGE-2152014).
The author received partial support from Jun Kitagawa's National Science Foundation grant (DMS-2246606).}}
}

\author{
  Kiyuob Jung\,\orcidlink{0009-0007-5075-4061}\thanks{Department of Mathematics, Michigan State University, East Lansing, MI 48824 USA (\email{kyjung@msu.edu}).}
}

\usepackage{amsopn}

%% file: reference.bib
@Book{1970_Rockafellar_BOOK,
  author     = {Rockafellar, R. Tyrrell},
  title      = {Convex analysis},
  pages      = {xviii+451},
  publisher  = {Princeton University Press, Princeton, NJ},
  series     = {Princeton Mathematical Series},
  volume     = {No. 28},
  mrclass    = {26.52 (46.00)},
  mrnumber   = {274683},
  mrreviewer = {Ky\ Fan},
  year       = {1970},
}

@Article{2018_Boyd,
  author  = {Boyd, Stephen and Mutapcic, Almir and Duchi, John},
  title   = {Stochastic Subgradient Methods},
  url     = {https://web.stanford.edu/class/ee364b/lectures/stoch_subgrad_notes.pdf},
  journal = {lecture notes of EE364b, Stanford University},
  year    = {2018},
}

@Book{2018_Nesterov_BOOK,
  author     = {Nesterov, Yurii},
  title      = {Lectures on convex optimization},
  doi        = {10.1007/978-3-319-91578-4},
  isbn       = {978-3-319-91577-7; 978-3-319-91578-4},
  pages      = {xxiii+589},
  publisher  = {Springer, Cham},
  series     = {Springer Optimization and Its Applications},
  volume     = {137},
  mrclass    = {90-01 (90C25)},
  mrnumber   = {3839649},
  mrreviewer = {Giorgio Giorgi},
  year       = {2018},
}

@Book{2006_Ruszczynski_BOOK,
  author    = {Ruszczy\'nski, Andrzej},
  title     = {Nonlinear optimization},
  doi       = {10.1515/9781400841059},
  isbn      = {978-0-691-11915-1; 0-691-11915-5},
  pages     = {xiv+448},
  publisher = {Princeton University Press, Princeton, NJ},
  mrclass   = {90-02 (90-01 90C30)},
  mrnumber  = {2199043},
  year      = {2006},
}

@Book{1998_Rockafellar,
  author     = {Rockafellar, R. Tyrrell and Wets, Roger J.-B.},
  title      = {Variational analysis},
  doi        = {10.1007/978-3-642-02431-3},
  isbn       = {3-540-62772-3},
  pages      = {xiv+733},
  publisher  = {Springer-Verlag, Berlin},
  volume     = {317},
  mrclass    = {49-02 (46N10 47N10 49J52 49K40 90C30)},
  mrnumber   = {1491362},
  mrreviewer = {Francis\ H.\ Clarke},
  year       = {1998},
}

@Book{2006_Mordukhovich_BOOK,
  author     = {Mordukhovich, Boris S.},
  title      = {Variational analysis and generalized differentiation. {I}},
  doi        = {10.1007/3-540-31247-1},
  isbn       = {978-3-540-25437-9; 3-540-25437-4},
  pages      = {xxii+579},
  publisher  = {Springer-Verlag, Berlin},
  series     = {Grundlehren der mathematischen Wissenschaften},
  volume     = {330},
  mrclass    = {49-02 (49J52 49J53 49K27 49K40 90C31)},
  mrnumber   = {2191744},
  mrreviewer = {Lionel\ Thibault},
  year       = {2006},
}

@Book{2023_Ryu_BOOK,
  author     = {Ryu, Ernest K. and Yin, Wotao},
  title      = {Large-scale convex optimization---algorithms \& analyses via monotone operators},
  doi        = {10.1017/9781009160865},
  isbn       = {978-1-009-16085-8},
  pages      = {xiv+303},
  publisher  = {Cambridge University Press, Cambridge},
  mrclass    = {90-02 (47H05 90C06 90C25)},
  mrnumber   = {4496334},
  mrreviewer = {Driss\ Mentagui},
  year       = {2023},
}

@Book{2025_Niculescu,
  author    = {Niculescu, Constantin P. and Persson, Lars-Erik},
  title     = {Convex functions and their applications},
  doi       = {10.1007/978-3-031-71967-7},
  edition   = {3rd},
  isbn      = {978-3-031-71966-0; 978-3-031-71967-7},
  pages     = {xxii+494},
  publisher = {Springer, Cham},
  series    = {CMS/CAIMS Books in Mathematics},
  volume    = {14},
  mrclass   = {26-01 (26A51 26B25 26D15 46A55 52A40 52A41 90C25)},
  mrnumber  = {4888856},
  year      = {2025},
}

@Book{2017_Kingma,
  author    = {Diederik P. Kingma and Jimmy Ba},
  title     = {Adam: A Method for Stochastic Optimization},
  publisher = {preprint, arXiv:1412.6980 [cs.LG]},
  url       = {https://arxiv.org/abs/1412.6980},
  year      = {2017},
}

@Book{2026a_Jung,
  author    = {Jung, Kiyuob},
  title     = {Nonlinear numerical schemes using specular differentiation for initial value problems of first-order ordinary differential equations},
  publisher = {preprint, arXiv:2601.09900 [math.NA]},
  url       = {https://arxiv.org/abs/2601.09900},
  year      = {2026},
}

@Article{1970_Goldfarb,
  author   = {Goldfarb, Donald},
  title    = {A family of variable-metric methods derived by variational means},
  doi      = {10.2307/2004873},
  issn     = {0025-5718,1088-6842},
  pages    = {23--26},
  volume   = {24},
  fjournal = {Mathematics of Computation},
  journal  = {Math. Comp.},
  mrclass  = {65.30},
  mrnumber = {258249},
  year     = {1970},
}

@Book{1985_Shor_BOOK,
  author     = {Shor, N. Z.},
  title      = {Minimization methods for nondifferentiable functions},
  doi        = {10.1007/978-3-642-82118-9},
  isbn       = {3-540-12763-1},
  pages      = {viii+162},
  publisher  = {Springer-Verlag, Berlin},
  series     = {Springer Series in Computational Mathematics},
  volume     = {3},
  mrclass    = {90C30 (49D07 65K10)},
  mrnumber   = {775136},
  mrreviewer = {K. G. Murty},
  year       = {1985},
}

@Article{2005_Zou,
  author   = {Zou, Hui and Hastie, Trevor},
  title    = {Regularization and variable selection via the elastic net},
  doi      = {10.1111/j.1467-9868.2005.00503.x},
  issn     = {1369-7412,1467-9868},
  number   = {2},
  pages    = {301--320},
  volume   = {67},
  fjournal = {Journal of the Royal Statistical Society. Series B. Statistical Methodology},
  journal  = {J. R. Stat. Soc. Ser. B Stat. Methodol.},
  mrclass  = {62M20},
  mrnumber = {2137327},
  year     = {2005},
}

@Article{1970_Broyden,
  author     = {Broyden, C. G.},
  title      = {The convergence of a class of double-rank minimization algorithms. {II}. {T}he new algorithm},
  issn       = {0020-2932},
  pages      = {222--231},
  volume     = {6},
  fjournal   = {Journal of the Institute of Mathematics and its Applications},
  journal    = {J. Inst. Math. Appl.},
  mrclass    = {65K05},
  mrnumber   = {433870},
  mrreviewer = {Y.\ Cherruault},
  year       = {1970},
}

@Book{2022_Foucart_BOOK,
  author    = {Foucart, Simon},
  title     = {Mathematical pictures at a data science exhibition},
  doi       = {10.1017/9781009003933},
  isbn      = {978-1-316-51888-5; 978-1-009-00185-4},
  pages     = {xx+318},
  publisher = {Cambridge University Press, Cambridge},
  mrclass   = {62-01 (62R07 90-01 94-01)},
  mrnumber  = {4394328},
  year      = {2022},
}

@Article{1970_Shanno,
  author     = {Shanno, D. F.},
  title      = {Conditioning of quasi-{N}ewton methods for function minimization},
  doi        = {10.2307/2004840},
  issn       = {0025-5718,1088-6842},
  pages      = {647--656},
  volume     = {24},
  fjournal   = {Mathematics of Computation},
  journal    = {Math. Comp.},
  mrclass    = {90.58},
  mrnumber   = {274029},
  mrreviewer = {J.\ S.\ Kowalik},
  year       = {1970},
}

@Article{1970_Fletcher,
  author  = {Fletcher, R.},
  title   = {A new approach to variable metric algorithms},
  doi     = {10.1093/comjnl/13.3.317},
  issn    = {0010-4620},
  number  = {3},
  pages   = {317-322},
  volume  = {13},
  journal = {The Computer Journal},
  month   = {01},
  year    = {1970},
}

@Book{1990_Clarke_BOOK,
  author    = {Clarke, F. H.},
  title     = {Optimization and nonsmooth analysis},
  doi       = {10.1137/1.9781611971309},
  edition   = {2nd},
  isbn      = {0-89871-256-4},
  pages     = {xii+308},
  publisher = {Society for Industrial and Applied Mathematics (SIAM), Philadelphia, PA},
  series    = {Classics in Applied Mathematics},
  volume    = {5},
  mrclass   = {49-02 (01A75 49J52 58C20 90C48)},
  mrnumber  = {1058436},
  year      = {1990},
}

@Article{2014_Boyd,
  author  = {Boyd, Stephen},
  title   = {Subgradient methods},
  url     = {https://web.stanford.edu/class/ee364b/lectures/subgrad_method_notes.pdf},
  journal = {lecture notes of EE364b, Stanford University},
  year    = {2014},
}

@Book{2004_Boyd_BOOK,
  author     = {Boyd, Stephen and Vandenberghe, Lieven},
  title      = {Convex optimization},
  doi        = {10.1017/CBO9780511804441},
  isbn       = {0-521-83378-7},
  pages      = {xiv+716},
  publisher  = {Cambridge University Press, Cambridge},
  mrclass    = {90-01 (90C05 90C25 90C46 90C51)},
  mrnumber   = {2061575},
  mrreviewer = {Dan Butnariu},
  year       = {2004},
}

@Book{2006_Nocedal_BOOK,
  author    = {Nocedal, Jorge and Wright, Stephen J.},
  title     = {Numerical optimization},
  doi       = {10.1007/978-0-387-40065-5},
  edition   = {2nd},
  isbn      = {978-0387-30303-1; 0-387-30303-0},
  pages     = {xxii+664},
  publisher = {Springer, New York},
  series    = {Springer Series in Operations Research and Financial Engineering},
  mrclass   = {90-01 (49Mxx 65K05 90-02 90C30)},
  mrnumber  = {2244940},
  year      = {2006},
}

@Book{2026s_Jung_SIAM,
  author    = {Jung, Kiyuob},
  title     = {{specular-differentiation}},
  doi       = {10.5281/zenodo.18246734},
  publisher = {v1.2.0, Zenodo},
  url       = {https://github.com/kyjung2357/specular-differentiation},
  year      = {2026},
}

@Book{1998_Shor_BOOK,
  author     = {Shor, Naum Z.},
  title      = {Nondifferentiable optimization and polynomial problems},
  doi        = {10.1007/978-1-4757-6015-6},
  isbn       = {0-7923-4997-0},
  pages      = {xviii+394},
  publisher  = {Kluwer Academic, Dordrecht},
  series     = {Nonconvex Optimization and its Applications},
  volume     = {24},
  mrclass    = {90-02 (49J52 90C30)},
  mrnumber   = {1620179},
  mrreviewer = {Bernd\ Luderer},
  year       = {1998},
}

@Article{2020_SciPy,
  author  = {Virtanen, Pauli and Gommers, Ralf and Oliphant, Travis E. and Haberland, Matt and Reddy, Tyler and Cournapeau, David and Burovski, Evgeni and Peterson, Pearu and Weckesser, Warren and Bright, Jonathan and {van der Walt}, St{\'e}fan J. and Brett, Matthew and Wilson, Joshua and Millman, K. Jarrod and Mayorov, Nikolay and Nelson, Andrew R. J. and Jones, Eric and Kern, Robert and Larson, Eric and Carey, C J and Polat, {\.I}lhan and Feng, Yu and Moore, Eric W. and {VanderPlas}, Jake and Laxalde, Denis and Perktold, Josef and Cimrman, Robert and Henriksen, Ian and Quintero, E. A. and Harris, Charles R. and Archibald, Anne M. and Ribeiro, Ant{\^o}nio H. and Pedregosa, Fabian and {van Mulbregt}, Paul and {SciPy 1.0 Contributors}},
  title   = {{{SciPy} 1.0: Fundamental Algorithms for Scientific Computing in Python}},
  doi     = {10.1038/s41592-019-0686-2},
  pages   = {261--272},
  volume  = {17},
  adsurl  = {https://ui.adsabs.harvard.edu/abs/2020NatMe..17..261V},
  journal = {Nature Methods},
  year    = {2020},
}

@Book{2026b_Jung,
  author    = {Jung, Kiyuob},
  title     = {Specular differentiation in normed vector spaces: {Q}uasi-{M}ean {V}alue and {Q}uasi-{Fermat} {T}heorems},
  publisher = {preprint, arXiv:2601.10950 [math.CA]},
  url       = {https://arxiv.org/abs/2601.10950},
  year      = {2026},
}

@InProceedings{2024_PyTorch,
  author    = {Ansel, Jason and Yang, Edward and He, Horace and Gimelshein, Natalia and Jain, Animesh and Voznesensky, Michael and Bao, Bin and Bell, Peter and Berard, David and Burovski, Evgeni and Chauhan, Geeta and Chourdia, Anjali and Constable, Will and Desmaison, Alban and DeVito, Zachary and Ellison, Elias and Feng, Will and Gong, Jiong and Gschwind, Michael and Hirsh, Brian and Huang, Sherlock and Kalambarkar, Kshiteej and Kirsch, Laurent and Lazos, Michael and Lezcano, Mario and Liang, Yanbo and Liang, Jason and Lu, Yinghai and Luk, CK and Maher, Bert and Pan, Yunjie and Puhrsch, Christian and Reso, Matthias and Saroufim, Mark and Siraichi, Marcos Yukio and Suk, Helen and Suo, Michael and Tillet, Phil and Wang, Eikan and Wang, Xiaodong and Wen, William and Zhang, Shunting and Zhao, Xu and Zhou, Keren and Zou, Richard and Mathews, Ajit and Chanan, Gregory and Wu, Peng and Chintala, Soumith},
  booktitle = {29th ACM International Conference on Architectural Support for Programming Languages and Operating Systems, Volume 2 (ASPLOS '24)},
  title     = {{PyTorch 2: Faster Machine Learning Through Dynamic Python Bytecode Transformation and Graph Compilation}},
  doi       = {10.1145/3620665.3640366},
  publisher = {ACM},
  year      = {2024},
}

@Book{2007_Schirotzek_BOOK,
  author     = {Schirotzek, Winfried},
  title      = {Nonsmooth analysis},
  doi        = {10.1007/978-3-540-71333-3},
  isbn       = {978-3-540-71332-6; 3-540-71332-8},
  pages      = {xii+373},
  publisher  = {Springer, Berlin},
  series     = {Universitext},
  mrclass    = {49-01 (49J52 49J53 90C48)},
  mrnumber   = {2330778},
  mrreviewer = {G\'erard\ Lebourg},
  year       = {2007},
}

@Book{1987_Polyak_BOOK,
  author    = {Polyak, Boris T.},
  title     = {Introduction to optimization},
  isbn      = {0-911575-14-6},
  pages     = {xxvii+438},
  publisher = {Optimization Software, Publications Division, New York},
  series    = {Translations Series in Mathematics and Engineering},
  mrclass   = {49-01 (65Kxx 90Cxx)},
  mrnumber  = {1099605},
  year      = {1987},
}
